\documentclass[a4paper,11pt]{article}

\usepackage{amssymb,amsmath,amsfonts}
\usepackage{graphicx,color,enumitem}
\usepackage{amsthm} 
\usepackage{bm}
\usepackage{geometry}
\usepackage[colorinlistoftodos, textwidth=4cm, shadow]{todonotes}
\RequirePackage[colorlinks,citecolor=blue,urlcolor=blue]{hyperref}
	
\usepackage{a4wide}
\usepackage[english]{babel}
\usepackage{caption}

\usepackage{appendix,authblk}

\newtheorem{theorem}{Theorem}[section]
\newtheorem{proposition}[theorem]{Proposition}
\newtheorem{lemma}[theorem]{Lemma}
\newtheorem{definition}[theorem]{Definition}
\newtheorem{remark}[theorem]{Remark}

\newtheorem{corollary}[theorem]{Corollary}
\newtheorem{assumption}{Assumption}[section]

\newcommand{\C}{{\mathbb C}}

\newcommand{\id}{{\rm id}}

\numberwithin{equation}{section}
\def\be{\begin{align}}
\def\ee{\end{align}}
\def\b*{\begin{eqnarray*}}
\def\e*{\end{eqnarray*}}

\def\be{\begin{eqnarray}}
\def\ee{\end{eqnarray}}
\def\beq{\begin{equation}}
\def\eeq{\end{equation}}
\def\b*{\begin{eqnarray*}}
\def\e*{\end{eqnarray*}}
\def\bi{\begin{itemize}}
\def\ei{\end{itemize}}

\def \1{{\bf 1}}

\def\={\;=\;}

\def \E{\mathbb{E}}
\def \F{\mathbb{F}}

\def \P{\mathbb{P}}

\def \R{\mathbb{R}}

\def\Gc{{\cal G}}

\title{Multi-factor approximation of rough volatility models}
\date{\today}

\author[1,2]{Eduardo {Abi Jaber}  \thanks{abijaber@ceremade.dauphine.fr}}
\author[3]{Omar {El Euch}\thanks{omar.el-euch@polytechnique.edu}}
\affil[1]{Universit\'e Paris-Dauphine, PSL Research University, CNRS, UMR [7534], CEREMADE, 75016 Paris, France.}
\affil[2]{AXA Investment Managers,  Multi Asset Client Solutions, Quantitative Research, \break
	6 place de la Pyramide, 92908 Paris - La D\'efense, France.}

\affil[3]{CMAP, Ecole Polytechnique, Paris. }

\begin{document}

	\maketitle

	\begin{abstract}
	\noindent Rough volatility models are very appealing because of their remarkable fit of both historical and implied volatilities. However, due to the non-Markovian and non-semimartingale nature of the volatility process, there is no simple way to simulate efficiently such models, which makes risk management of derivatives an intricate task. In this paper, we design tractable multi-factor stochastic volatility models approximating rough volatility models and enjoying a Markovian structure. Furthermore, we apply our procedure to the specific case of the rough Heston model. This in turn enables us to derive a numerical method for solving fractional Riccati equations appearing in the characteristic function of the log-price in this setting.
		
		
	\end{abstract}


\section{Introduction}  
Empirical studies of  a very wide range of assets volatility time-series in  \cite{gatheral2014volatility} have shown that the dynamics of  the log-volatility are close to that of a fractional Brownian motion $W^H$ with a small Hurst parameter $H$ of order $0.1$. Recall that a fractional Brownian motion $W^H$ can be built from a two-sided Brownian motion thanks to the Mandelbrot-van Ness representation
\begin{equation*}
W_t^H =\frac{1}{\Gamma(H+1/2)} \int_0^t (t-s)^{H-\frac{1}{2}} dW_s + \frac{1}{\Gamma(H+1/2)} \int_{-\infty}^0 \big((t-s)^{H-\frac{1}{2}} - (-s)^{H-\frac{1}{2}}\big)dW_s.
\end{equation*}
The fractional kernel $(t-s)^{H- \frac{1}{2}}$ is behind the $H - \varepsilon$ H\"older regularity of the volatility for any $\varepsilon > 0$. For small values of the Hurst parameter $H$, as observed empirically, stochastic volatility models involving the fractional kernel are called rough volatility models. \\

\noindent Aside from modeling historical volatility dynamics, rough volatility models reproduce accurately with very few parameters the behavior of the implied volatility surface, see \cite{bayer2016pricing, euch2017roughening}, especially the at-the-money skew, see \cite{fukasawa2011asymptotic}. Moreover, microstructural foundations of rough volatility are studied in \cite{eleuch2016micro, jaisson2016rough}. \\

\noindent In this paper, we are interested in a class of rough volatility models where the dynamics of the asset price $S$ and its stochastic variance $V$ are given by 
\begin{equation}\label{roughPrice} 
 dS_t = S_t \sqrt{V_t} dW_t,  \quad S_0>0,
\end{equation}
\begin{equation}\label{roughVol}
 V_t = V_0 +\frac{1}{\Gamma(H +\frac{1}{2})} \int_0^t (t-u)^{H - \frac{1}{2}}  (\theta(u) - \lambda V_u)  du  + \frac{1}{\Gamma(H  + \frac{1}{2})} \int_0^t (t-u)^{H - \frac{1}{2}}  \sigma(V_u)  dB_u, 
\end{equation}
for all $t \in [0,T]$, on some filtered probability space $(\Omega, {\cal F}, \F, \P)$. Here $T$ is a positive time horizon, the parameters $\lambda$ and $V_0$ are non-negative, $H \in (0, 1/2)$ is the Hurst parameter, $\sigma$ is a continuous function and $W = \rho B + \sqrt{1-\rho^2} B^{\perp}$ with $(B,B^{\perp})$ a two-dimensional $\mathbb F$-Brownian motion and $\rho \in [-1,1]$. Moreover, $\theta$ is a deterministic mean reversion level allowed to be time-dependent to fit the market forward variance curve $(\E[V_t])_{t \leq T}$ as explained in Section \ref{SectionDef} and in \cite{euch2017perfect}. Under some general assumptions, we establish in Section \ref{SectionDef} the existence of a weak non-negative solution to  the fractional stochastic {integral} equation in \eqref{roughVol}   exhibiting $H - \varepsilon$ H\"older regularity for any $\varepsilon > 0$. Hence, this class of models is a natural rough extension of classical stochastic volatility models where the fractional kernel is introduced in the drift and stochastic part of the variance process $V$. Indeed, when $H = 1/2$, we recover classical stochastic volatility models where the variance process is a standard diffusion.\\

\noindent Despite the fit to the historical and implied volatility, some difficulties are encountered in practice for  the simulation of rough volatility models and for pricing and  hedging derivatives with them. In fact, due to the introduction of the fractional kernel, we lose the Markovian and semimartingale structure. In order to overcome theses difficulties, we approximate these models by simpler ones that we can use in practice. \\

\noindent In \cite{euch2016characteristic, eleuch2016micro, euch2017perfect}, the rough Heston model (which corresponds to the case of $\sigma(x) = \nu \sqrt{x}$) is built as a limit of microscopic Hawkes-based price models. This allowed the understanding of the microstructural foundations of rough volatility and also led to the formula of the characteristic function of the log-price. Hence, the Hawkes approximation enabled us to solve the pricing and hedging under the rough Heston model. However, this approach is specific to the rough Heston case and can not be extended to an arbitrary rough volatility model of the form \eqref{roughPrice}-\eqref{roughVol}.\\

\noindent Inspired by the works of \cite{papier2, CC98, CCM00, HS15, M11}, we provide a natural Markovian approximation for the class of rough volatility models \eqref{roughPrice}-\eqref{roughVol}. The main idea is to write the fractional kernel $K(t) = \frac{t^{H - \frac{1}{2}}}{\Gamma(H + 1/2)}$ as a Laplace transform of a positive measure $\mu$
\begin{equation} \label{Laplace}
K(t) = \int_{0}^{\infty} e^{- \gamma t} \mu(d\gamma) ; \quad \mu(d\gamma) = \frac{\gamma^{- H - \frac12}}{ \Gamma(H+1/2) \Gamma(1/2 - H)} d\gamma.
\end{equation}
We then approximate $\mu$ by a finite sum of Dirac measures $\mu^n = \sum_{i=1}^n c_i^n \delta_{\gamma_i^n}$ with positive weights $(c_i^n)_{1 \leq i \leq n}$ and  mean reversions $(\gamma_i^n)_{1 \leq i \leq n}$, for $n \geq 1$. This in turn yields an approximation of the fractional kernel by a sequence of smoothed kernels $(K^n)_{n \geq 1}$ given by 
 \begin{equation*}
 K^n(t) = \sum_{i = 1}^n c^n_i e^{- \gamma_i^n t}, \quad  n \geq 1.
\end{equation*}
This leads to a multi-factor stochastic volatility model $(S^n, V^n) = (S^n_t, V^n_t)_{t \leq T}$, which is Markovian with respect to the spot price and $n$ variance factors $(V^{n,i})_{1 \leq i \leq n}$ and is defined as follows
\begin{equation} \label{multifactorV}  
dS^n_t = S^n_t \sqrt{V^n_t} dW_t , \quad  V^n_t = g^n(t) + \sum_{i=1}^n c_i^n V^{n, i}_t, 
\end{equation}
where 
$$ dV^{n,i}_t = (- \gamma^n_i V^{n,i}_t - \lambda V^n_t) dt + \sigma(V^n_t) dB_t , $$
and $g^n(t) = V_0 + \int_0^t K^n(t-s) \theta(s) ds $ with the initial conditions $S^n_0 = S_0$ and $V^{n, i}_0 = 0$. Note that the factors $(V^{n,i})_{1 \leq i \leq n}$ share the same dynamics except that they mean revert at different speeds $(\gamma_i^n)_{1 \leq i \leq n}$. Relying on existence results of stochastic Volterra equations in \cite{papier2, ALP17}, we provide in Theorem \ref{Yamada} the strong existence and uniqueness of the model $(S^n, V^n)$, under some general conditions. Thus the approximation \eqref{multifactorV} is uniquely well-defined. We can therefore deal with simulation, pricing and hedging problems under these multi-factor models by using standard methods developed for stochastic volatility models. \\

\noindent Theorem \ref{mainResult}, which is the main result of this paper, establishes the convergence of the multi-factor approximation sequence $(S^n, V^n)_{n \geq 1}$ to the rough volatility model $(S,V)$ in \eqref{roughPrice}-\eqref{roughVol} when the number of factors $n$ goes to infinity, under a suitable choice of the weights and mean reversions $(c_i^n, \gamma_i^n)_{1 \leq i \leq n}$ . This convergence is obtained from a general result about stability of stochastic Volterra equations derived in Section \ref{stabilitySec}. \\

\noindent In \cite{ALP17, euch2016characteristic, euch2017perfect}, the characteristic function of the log-price for the specific case of the rough Heston model is obtained in terms of a solution of a fractional Riccati equation. We highlight in Section \ref{multiSchemeSec} that the corresponding multi-factor approximation \eqref{multifactorV} inherits a similar affine structure as in the rough Heston model. More precisely, it displays the same characteristic function formula involving a $n$-dimensional classical Riccati equation instead of the fractional one. This suggests solving numerically the fractional Riccati equation by approximating it through a $n$-dimensional classical Riccati equation with large $n$, see Theorem \ref{convergenceRiccati}. In Section \ref{illustrationSec}, we discuss the accuracy and complexity of this numerical method and compare it to the Adams scheme, see \cite{diethelm2002predictor,diethelm2004detailed,diethelm1998fracpece, euch2016characteristic}.\\

\noindent The paper is organized as follows. In Section \ref{SectionDef}, we define the class of rough volatility models \eqref{roughPrice}-\eqref{roughVol} and discuss the existence of such models. Then, in Section \ref{mainSec}, we build a sequence of multi-factor stochastic volatility models of the form of \eqref{multifactorV} and show its convergence to a rough volatility model. By applying this approximation to the specific case of the rough Heston model, we obtain a numerical method for computing solutions of fractional Riccati equations that is discussed in Section \ref{rHestonSec}. Finally, some proofs are relegated to Section \ref{proofSec} and some useful technical results are given in an Appendix.
	
\section{A definition of rough volatility models}\label{SectionDef}
We provide in this section the precise definition of rough volatility models given by \eqref{roughPrice}-\eqref{roughVol}. We discuss the existence of such models and more precisely of a non-negative solution of the fractional stochastic integral equation \eqref{roughVol}. The existence of an unconstrained weak solution $V=(V_t)_{t \leq T}$ is guaranteed by Corollary \ref{weakExistenceFrac} in the Appendix when $\sigma$ is a continuous function with linear growth and $\theta$ satisfies the condition
\begin{equation} \label{cond1}
\forall \varepsilon > 0 , \quad \exists C_\varepsilon>0;  \quad \forall u \in (0,T] \quad  |\theta(u)| \leq C_\varepsilon u^{-\frac12 - \varepsilon}.
\end{equation}
Furthermore, the paths of $V$ are H\"older continuous of any order strictly less than $H$ and
\begin{equation} \label{EV}
\sup_{t \in [0,T]}\E[|V_t|^p] < \infty , \quad p>0.
\end{equation}
Moreover using Theorem \ref{existencePositive} together with Remarks \ref{positiveFrac} and \ref{postiiveG_rem} in the Appendix\footnote{Theorem \ref{existencePositive} is used here with the fractional kernel $K(t) = \frac{t^{H-\frac12}}{\Gamma(H+1/2)}$ together with $b(x) = - \lambda x$ and $g(t) = V_0 + \int_0^t K(t-u)  \theta(u) du.$}, the existence of a non-negative continuous process $V$ satisfying \eqref{roughVol} is obtained under the additional conditions of non-negativity of $V_0$ and $\theta$  
and $\sigma(0) = 0$. We can therefore introduce the following class of rough volatility models.

\begin{definition} \label{roughVolModel}(Rough volatility models) We define a rough volatility model by any $\R\times \R_+$-valued continuous process $(S, V) = (S_t, V_t)_{t \leq T}$ satisfying
$$ dS_t = S_t \sqrt{V_t} dW_t ,$$
\begin{equation*}
 V_t = V_0 +\frac{1}{\Gamma(H+1/2)} \int_0^t (t-u)^{H - \frac12}  (\theta(u) - \lambda V_u)  du  + \frac{1}{\Gamma(H+1/2)} \int_0^t (t-u)^{H -  \frac12}  \sigma(V_u)  dB_u,
\end{equation*}
on a filtred probability space $(\Omega, {\cal F}, \F, \P)$ with non-negative initial conditions $(S_0, V_0)$. Here $T$ is a positive time horizon, the parameter $\lambda$ is non-negative, $H \in (0, 1/2)$ is the Hurst parameter and $W = \rho B + \sqrt{1-\rho^2} B^{\perp}$ with $(B,B^{\perp})$ a two-dimensional $\mathbb F$-Brownian motion and $\rho \in [-1,1]$. Moreover, to guarantee the existence of such model, $\sigma : \R \mapsto \R$ is assumed continuous with linear growth such that $\sigma(0) = 0$ and  $\theta: [0,T] \mapsto \R$ is a deterministic non-negative function satisfying \eqref{cond1}.
\end{definition}
As done in \cite{euch2017perfect}, we allow the mean reversion level $\theta$ to be time dependent in order to be consistent with the market forward variance curve. More precisely,  the following result shows that the mean reversion level $\theta$ can be written as a functional of the forward variance curve $(\E[V_t])_{t \leq T}$.

\begin{proposition} Let $(S,V)$ be a rough volatility model given by Definition \ref{roughVolModel}. Then, $(\E[V_t])_{t \leq T}$ is linked to $\theta$ by the following formula
\begin{equation}\label{EV1}
\E[V_t] = V_0 + \int_0^t (t-s)^{\alpha - 1} E_{\alpha} (-\lambda (t-s)^\alpha) \theta(s) ds,  \quad t \in [0,T], 
\end{equation}
where $\alpha = H+1/2$ and $E_\alpha(x) = \sum_{k \geq 0} \frac{x^k}{\Gamma(\alpha(k + 1))}$ is the Mittag-Leffler function. Moreover, $(\E[V_t])_{t \leq T}$ admits a fractional derivative\footnote{Recall that the fractional derivative of order $\alpha \in (0, 1)$ of a function $f$ is given by $\frac{d}{dt} \int_0^t \frac{(t-s)^{-\alpha}}{\Gamma(1-\alpha)} f(s) ds$ whenever this expression  is well defined.} of order $\alpha$ at each time $t \in (0, T]$ and
\begin{equation}\label{EV2} 
\theta(t) = D^{\alpha} (\E[V_.] - V_0)_t + \lambda \E[V_t],  \quad t \in (0, T].
\end{equation}
\end{proposition}
\begin{proof} Thanks to \eqref{EV} together with Fubini theorem, $t \mapsto \E[V_t]$ solves the following fractional linear integral equation
\begin{equation}\label{EV0} 
 \E[V_t] = V_0 + \frac{1}{\Gamma(H+1/2)}\int_0^t (t-s)^{H-\frac12} (\theta(s) - \lambda \E[V_s])ds, \quad t \in [0,T],
 \end{equation}
yielding \eqref{EV1} by Theorem \ref{volterraLinear} and Remark \ref{resolventFrac} in the Appendix. Finally, \eqref{EV2} is obviously obtained from \eqref{EV0}.
\end{proof}
Finally, note that {uniqueness} of the fractional stochastic {integral} equation \eqref{roughVol} is a {difficult} problem. Adapting the proof in \cite{mytnik2015uniqueness}, we can prove  pathwise uniqueness when $\sigma$ is $\eta$-H\"older continuous with $\eta \in (1/(1+2H), 1]$. This result  does not cover the square-root case, i.e.~$\sigma(x) = \nu \sqrt{x}$, for which  weak uniqueness  has been established in \cite{papier2, ALP17, mytnik2015uniqueness}.

\section{Multi-factor approximation of rough volatility models}\label{mainSec}
 
Thanks to the small H\"older regularity of the variance process, models of Definition \ref{roughVolModel} are able to reproduce the rough behavior of the volatility observed in a wide range of assets. However, the fractional kernel forces the variance process  to leave both the semimartingale and Markovian worlds, which makes numerical approximation procedures a difficult and challenging task in practice.
The aim of this section is to construct a tractable and satisfactory Markovian approximation of any rough volatility model $(S, V)$ of Definition \ref{roughVolModel}. Because $S$ is entirely determined by $(\int_0^\cdot V_s ds,\int_0^{\cdot} \sqrt{V_s}dW_s)$, it suffices to construct a suitable approximation of the variance process $V$.  This is done  by smoothing the fractional kernel.\\

More precisely, denoting by $K(t) = \frac{ t^{H -  \frac12}}{\Gamma(H+1/2)} $, the fractional stochastic integral equation \eqref{roughVol} reads
$$ V_t = V_0 + \int_0^t K(t-s) \left( (\theta(s) - \lambda V_s) ds + \sigma(V_s)dB_s \right) , $$
which is a stochastic Volterra equation. Approximating the fractional kernel $K$ by a sequence of smooth kernels $(K^n)_{n \geq 1}$, one would expect the convergence of the following corresponding sequence of stochastic Volterra equations
	 $$ V_t^n = V_0 + \int_0^t K^n(t-s) \left( (\theta(s) - \lambda V_s^n) ds + \sigma(V_s^n)dB_s \right),  \quad n \geq 1, $$
	 to the fractional one.\\
	
	 The argument of this section runs as follows. First, exploiting the identity    \eqref{Laplace}, we construct a family of potential candidates for $(K^n,V^n)_{n \geq 1}$ in Section \ref{S:KnVn} such that $V^n$ enjoys a Markovian structure. Second, we provide convergence conditions of  $(K^n)_{n \geq 1}$ to $K$ in $\mathbb L^2([0,T], \R)$ in Section \ref{choiceK}. Finally, the approximation result for the rough volatility model $(S,V)$ is established in Section \ref{S:conv} relying on an abstract stability result of stochastic Volterra equations  postponed to  Section \ref{stabilitySec} for sake of exposition. 
		
\subsection{Construction of the approximation} \label{S:KnVn}
In \cite{CC98, HS15, M11}, a Markovian representation of the fractional Brownian motion of Riemann-Liouville type is provided  by writing the fractional kernel $K(t) = \frac{ t^{H -  \frac12}}{\Gamma(H+1/2)} $ as a Laplace transform of a non-negative measure $\mu$ as in \eqref{Laplace}.
This representation is extended in \cite{papier2} for the Volterra square-root process. Adopting the same approach, we establish a similar representation for any solution of the fractional stochastic integral equation \eqref{roughVol} in terms of an infinite dimensional system of processes sharing the same Brownian motion and mean reverting at different speeds. Indeed by using the linear growth of $\sigma$ together with the stochastic Fubini theorem, see \cite{V:12}, we obtain that 
$$ V_t = g(t) + \int_0^{\infty} V^{\gamma}_t \mu(d\gamma),  \quad t \in [0, T], $$
with 
$$ dV^{\gamma}_t =  (- \gamma V^{\gamma}_t - \lambda V_t) dt + \sigma(V_t) dB_t, \quad V^\gamma_0 = 0 , \quad \gamma \geq 0,$$
and
\begin{equation} \label{casG}
g(t) = V_0 + \int_0^t K(t-s) \theta(s) ds.
\end{equation}

Inspired by \cite{CC98, CCM00}, we approximate the measure $\mu$ by a weighted sum of Dirac measures
$$ \mu^n  = \sum_{i = 1}^n c_i^n \delta_{\gamma^n_i} , \quad n \geq 1, $$
leading to the following approximation $V^n = (V^n_t)_{t \leq T}$ of the variance process $V$
\begin{equation}  \label{multifactorSDE}
V_t^n = g^n(t) + \sum_{i = 1}^n c_i^n V^{n, i}_t  , \quad t \in [0,T], 
\end{equation}
$$ d V^{n, i}_t = ( - \gamma_i^n V^{n, i}_t - \lambda V^n_t) dt +  \sigma(V^{n}_t) dB_t, \quad V^{n,i}_0 = 0, $$
where 
\begin{equation} \label{casG_n}
g^n(t) = V_0 + \int_0^t K^n(t-u) \theta(u) du ,
\end{equation}
and 
\begin{equation} \label{smoothedKernel}
K^n(t) = \sum_{i=1}^n c_i^n e^{-\gamma_i^n t}. 
\end{equation}
The choice of the positive weights $(c_i^n)_{1 \leq i \leq n} $ and mean reversions $(\gamma_i ^n)_{1 \leq i \leq n}$, which is crucial for the accuracy of the approximation, is studied in Section \ref{choiceK} below. Before proving the convergence of $(V^n)_{n \geq 1}$, we shall first discuss the existence and uniqueness of such processes. This is done by rewriting the stochastic equation \eqref{multifactorSDE} as a stochastic Volterra equation of the form
\begin{equation} \label{multifactorVolterra}
V^n_t = g^n(t) +  \int_0^t K^n(t-s) \left( - \lambda V^n_s ds +  \sigma(V^n_s) dB_s\right) , \quad t \in [0, T].
\end{equation}
The existence of a continuous non-negative weak solution $V^n$ is ensured by Theorem \ref{existencePositive} together with Remarks \ref{positiveFrac} and \ref{postiiveG_rem} in the Appendix\footnote{Theorem \ref{existencePositive} is used here with the smoothed kernel $K^n$ given by \eqref{smoothedKernel} together with $b(x) = - \lambda x$ and $g $ defined as in \eqref{casG}}, because $\theta$ and $V_0$ are non-negative and $\sigma(0) = 0$.  
Moreover, pathwise uniqueness of solutions to \eqref{multifactorVolterra} follows by adapting the standard arugments of \cite{YW71}, provided  a suitable H\"older continuity of $\sigma$, see Proposition \ref{T:Yamada} in the Appendix. Note that this extension is made possible due to the smoothness of the kernel $K^n$. For instance, this approach fails for the fractional kernel because of the singularity at zero. This leads us to the following result which establishes the strong existence and uniqueness of a non-negative solution of \eqref{multifactorVolterra} and equivalently of \eqref{multifactorSDE}. 
\begin{theorem} \label{Yamada} Assume that $\theta : [0,T] \mapsto \R$ is a deterministic non-negative function satisfying \eqref{cond1} and that $ \sigma : \R \mapsto \R $ is $\eta$-H\"older continuous with $\sigma(0) = 0$ and $\eta \in [1/2,1]$. Then, there exists  a unique strong non-negative solution $V^n = (V^n_t)_{t \leq T}$ to the stochastic Volterra equation \eqref{multifactorVolterra} for each $n \geq 1$. 
 \end{theorem}

Due to the uniqueness of  \eqref{multifactorSDE}, we obtain that $V^n$ is a Markovian process according to $n$ state variables $(V^{n,i})_{1 \leq i \leq n}$ that we call the factors of $V^n$. Moreover, $V^n$ being non-negative, it can model  a variance process. This leads to the following definition of multi-factor stochastic volatility models. 

\begin{definition} (Multi-factor stochastic volatility models). \label{multifactorModel} We define the following sequence of multi-factor  stochastic volatility models $(S^n,V^n)=(S^n_t, V^n_t)_{t \leq T}$ as the unique $\R \times \R_+$-valued strong solution of 
$$ dS^n_t = S^n_t \sqrt{V^n_t} dW_t, \quad  V^n_t = g^n(t) + \sum_{i = 1}^n c_i^n V^{n,i}_t,$$
with
$$ dV^{n,i}_t = (-\gamma_i^n V^{n,i}_t - \lambda V_t^n) dt + \sigma(V_t^n) dB_t, \quad V^{n,i}_0 = 0, \quad S^n_0 = S_0 > 0 , $$
on a filtered probability space $(\Omega, \cal F, \P, \mathbb F)$, where $\mathbb F$ is the canonical filtration a two-dimensional Brownian motion $(W, W^\perp)$ and $B = \rho W + \sqrt{1-\rho^2} W^\perp$ with $\rho \in [-1, 1]$. Here, the weights $(c_i^n)_{1 \leq i \leq n}$ and mean reversions $(\gamma_i^n)_{1 \leq i \leq n}$ are positive,  $\sigma : \R \mapsto \R$ is $\eta$-H\"older continuous such that $\sigma(0) = 0$, $\eta \in [1/2,1]$ and $g^n$ is given by \eqref{casG_n}, that is
$$g^n(t) = V_0 + \int_0^t K^n(t-s) \theta(s) ds, $$
with a non-negative initial variance $V_0$, a kernel $K^n$ defined as in \eqref{smoothedKernel} and a non-negative deterministic function $\theta : [0,T] \mapsto \R$ satisfying \eqref{cond1}. 
\end{definition}
Note that the strong existence and uniqueness of $(S^n, V^n)$ follows from Theorem \ref{Yamada}. This model is Markovian with $n + 1$ state variables which are  the spot price $S^n$ and the factors of the variance process $V^{n,i}$ for $ i \in \{1,\dots, n\}$.

\subsection{An approximation of the fractional kernel} \label{choiceK}
Relying on \eqref{multifactorVolterra}, we can see the process $V^n$ as an approximation of $V$, solution of \eqref{roughVol}, obtained by smoothing the fractional kernel $K(t) = \frac{t^{H - \frac12}}{\Gamma(H+1/2)}$ into  $K^n(t) = \sum_{i = 1}^n c_i^n e^{-\gamma_i^n t}$. Intuitively, we need to choose $K^n$ close to $K$ when $n$ goes to infinity, so that $(V^n)_{n \geq 1}$ converges to $V$. Inspired by \cite{CCM00}, we give in this section a condition on the weights $(c^n_i)_{1 \leq i \leq n}$ and mean reversion terms  $0 < \gamma_1^n <  ... < \gamma_n^n$ so that the following convergence  
$$ \left\| K^n - K\right\|_{2,T} \rightarrow 0 ,$$
holds as $n$ goes to infinity,
 where $\| \cdot \|_{2, T}$ is the usual $\mathbb L^2([0,T], \R)$ norm. Let $(\eta_i^n)_{0 \leq i \leq n}$ be auxiliary mean reversion terms such that $\eta_0^n=0$ and  $\eta_{i-1}^n \leq \gamma_i^n \leq \eta_i^n$ for $ i \in \{1,\dots, n\}$. Writing $K$ as the Laplace transform of $\mu$ as in \eqref{Laplace}, we obtain that
	$$\left\| K^n - K \right\|_{2,T}  \leq \int_{\eta_n^n}^\infty \| e^{-\gamma (\cdot)} \|_{2,T}\mu(d\gamma) + \sum_{i = 1}^n J_i^n, $$
	with $J_i^n = \| c_i^n e^{- \gamma_i^n (\cdot)} - \int_{\eta_{i-1}^n}^{\eta_i^n} e^{-\gamma (\cdot)} \mu(d\gamma) \|_{2, T} $. We start by dealing with the first term, 
\begin{equation*}  
  \int_{\eta_n^n}^\infty \| e^{-\gamma (\cdot)} \|_{2,T} \mu(d\gamma) =  \int_{\eta_n^n}^\infty \sqrt{\frac{1-e^{-2\gamma T}}{2\gamma}} \mu(d\gamma) \leq \frac{1}{ H  \Gamma(H+1/2) \Gamma(1/2-H) \sqrt{2}} (\eta^n_n)^{-H}.  
  \end{equation*}
	
	Moreover by choosing
	\begin{equation}\label{E:cgamma} 
	c_i^n = \int_{\eta_{i-1}^n}^{\eta_i^n} \mu(d\gamma), \quad \gamma_i^n = \frac{1}{c^n_i} \int_{\eta_{i-1}^n}^{\eta_i^n} \gamma \mu(d\gamma), \quad i \in \{1,\dots,n\},
	\end{equation}
	and using the Taylor-Lagrange inequality up to the second order, we obtain 
	\begin{equation}  \label{tayllor}
	\left| c_i^n e^{- \gamma_i^n t} - \int_{\eta_{i-1}^n}^{\eta_i^n} e^{-\gamma t} \mu(d\gamma) \right| \leq \frac{t^2}{2} \int_{\eta_{i-1}^n}^{\eta_i^n}(\gamma - \gamma_i^n)^2 \mu(d\gamma), \quad t \in [0,T]. 
	\end{equation}
	
	Therefore,
	$$ \sum_{i = 1}^n J_i^n \leq \frac{T^{5/2}}{2\sqrt5} \sum_{i=1}^n \int_{\eta_{i-1}^n}^{\eta_i^n} (\gamma_i^n-\gamma)^2 \mu(d\gamma)  .
$$
This leads to the following inequality
	\begin{equation*}
	\| K^n - K \|_{2,T}  \leq f_n^{(2)} \big((\eta_i)_{0 \leq i \leq n}\big)	,
	\end{equation*}
	where $f_n^{(2)}$ is a function of the auxiliary mean reversions defined by
	\begin{equation} \label{f_n_2}
	f_n^{(2)}((\eta_i^n)_{1 \leq i \leq n}) = \frac{T^{\frac52}}{2 \sqrt{5}}  \sum_{i = 1}^n \int_{\eta_{i-1}^n}^{\eta_i^n}(\gamma - \gamma_i^n)^2 \mu(d\gamma) + \frac{1}{ H  \Gamma(H+1/2) \Gamma(1/2-H) \sqrt{2}} (\eta^n_n)^{-H}.
	\end{equation}

	Hence, we obtain the convergence of $K^n$ to the fractional kernel under the following choice of weights and mean reversions.
	\begin{assumption} \label{factorsAssump} We assume  that the weights and mean reversions are given by \eqref{E:cgamma} such that $ \eta_0^n=0 < \eta_1^n < \ldots < \eta_n^n$ and 
	\begin{equation} \label{condFacteurs}
	\eta_n^n \rightarrow \infty , \quad  \sum_{i=1}^n \int_{\eta_{i-1}^n}^{\eta_i^n} (\gamma_i^n-\gamma)^2 \mu(d\gamma) \rightarrow 0,  
	\end{equation}
	as $n$ goes to infinity.
	\end{assumption}

	\begin{proposition} \label{convergenceK}  Fix $(c^n_i)_{1 \leq i\leq n}$ and $(\gamma_i^n)_{1 \leq i\leq n}$  as in Assumption \ref{factorsAssump} and   $K^n$ given by \eqref{smoothedKernel}, for all $n \geq 1$. Then, $(K^n)_{n \geq 1}$ converges in $\mathbb L^2[0,T]$ to the fractional kernel $K(t) = \frac{t^{H - 1/2}}{\Gamma(H+\frac12)} $ as $n$ goes to infinity.
	\end{proposition}

	There exists several choices of auxiliary factors such that condition \eqref{condFacteurs} is met. For instance, assume that $\eta_i^n = i \pi_n$ for each $i \in \{0,\dots, n\}$ such that $\pi_n >0$.  It follows from
	$$ \sum_{i = 1}^n \int_{\eta_{i-1}^n}^{\eta_i^n} (\gamma - \gamma_i)^2 \mu(d\gamma) \leq \pi_n^2 \int_0^{\eta_n^n} \mu(d\gamma) = \frac{1}{  (1/2-H) \Gamma(H+1/2) \Gamma(1/2-H)} \pi_n^{\frac52-H} n^{\frac12-H}, $$ 
	that \eqref{condFacteurs} is satisfied for 
	$$ \eta_n^n = n \pi_n \rightarrow \infty , \quad  \pi_n^{\frac52-H} n^{\frac12-H} \rightarrow 0 ,$$
	as $n$ tends to infinity. In this case,
	$$\|K^n - K\|_{2, T} \leq\frac{1}{H \Gamma(H+1/2) \Gamma(1/2-H) \sqrt{2}} \left((\eta_n^n)^{-H} + \frac{H T^{\frac52}}{\sqrt{10} (1/2-H)} \pi_n^{2} (\eta_n^n)^{\frac12-H} \right). $$
	This upper bound is minimal for 
	\begin{equation}\label{pasOptimal}
	 \pi_n = \frac{n^{-\frac15}}{T} \big(\frac{\sqrt{10} (1-2H)}{5 - 2H}\big)^{\frac25} ,
	\end{equation}
	and
	$$ \| K^n - K \|_{2,T} \leq C_H n^{-\frac{4H}5}, $$
	where  $C_H$ is a positive  constant that can be computed explicitly and that depends only on the Hurst parameter $H \in (0, 1/2)$.

	\begin{remark} Note that the kernel approximation in Proposition \ref{convergenceK} can be easily extended to any kernel of the form
	$$ K(t) = \int_{0}^{\infty} e^{- \gamma t} \mu(d\gamma), $$
	where $\mu$ is a non-negative measure such  that 
	\begin{equation*} 
	\int_0^{\infty} (1 \wedge \gamma^{-1/2}) \mu(d\gamma) < \infty.
	\end{equation*}
		\end{remark}

\subsection{Convergence result}	\label{S:conv}
We assume now that the weights and mean reversions of the multi-factor stochastic volatility model $(S^n, V^n)$ satisfy Assumption \ref{factorsAssump}. Thanks to Proposition \ref{convergenceK}, the smoothed kernel $K^n$ is close to the fractional one for large $n$. Because $V^n$ satisfies the stochastic Volterra equation \eqref{multifactorVolterra}, $V^n$ has to be close to $V$ and thus by passing to the limit, $(S^n, V^n)_{n\geq 1}$ should  converge to the rough volatility model $(S,V)$ of Definition \ref{roughVolModel} as $n$ goes large. This is the object of the next theorem, which is the main result of this paper.
\begin{theorem}\label{mainResult} Let $(S^n, V^n)_{n \geq 1}$ be a sequence of multi-factor stochastic volatility models given by Definition \ref{multifactorModel}. Then, under Assumption \ref{factorsAssump}, the family $(S^n, V^n)_{n \geq 1}$ is tight for the uniform topology and any point limit $(S, V)$ is a rough volatility model given by Definition \ref{roughVolModel}. 
\end{theorem}
Theorem \ref{mainResult}  states the convergence in law of $(S^n, V^n)_{n \geq 1}$ whenever the fractional stochastic integral equation \eqref{roughVol} admits a unique weak solution. 
In order to prove Theorem \ref{mainResult}, whose proof is in Section \ref{mainResultProof} below, a more general stability result for $d$-dimensional stochastic Volterra equations is established in the next subsection.

\subsection{Stability of stochastic Volterra equations} \label{stabilitySec}
	As mentioned above, Theorem \ref{mainResult} relies on the study of the stability of more general $d$-dimensional stochastic Volterra equations of the form 	
	\begin{equation} \label{generalDiffusion}
	X_t = g(t) + \int_0^t K(t-s) b(X_s) ds + \int_0^t K(t-s) \sigma(X_s) dW_s, \quad t \in [0,T],
	\end{equation}
	where $b: \mathbb R^d \rightarrow \mathbb R^d$, $\sigma : \mathbb R^d \rightarrow {\mathbb R}^{d \times m}$ are continuous and satisfy the linear growth condition,  $K \in \mathbb L^2([0,T], \R^{d \times d})$ admits a resolvent of the first kind $L$, see Appendix~\ref{resolvent1},  and $W$  is a $m$-dimensional Brownian motion on some filtered probability space $(\Omega, {\cal F}, \mathbb F ,\mathbb P )$. From Proposition \ref{weakExistence} in the Appendix,  $g: [0,T] \mapsto \R^d$ and $K \in \mathbb L^2([0,T], \R^{d \times d})$ should satisfy Assumption \ref{regularity}, that is
	\begin{equation} \label{regularity_eq} 
	|g(t+h) - g(t)|^2 + \int_0^h |K(s)|^2 ds + \int_0^{T-h} |K(h+s) - K(s)|^2 ds  \leq C h^{2\gamma}, 
	\end{equation}
	for any $t, h \geq 0$ with $t + h \leq T$ and for some positive constants $C$ and $\gamma$, to guarantee the weak existence of a continuous solution $X$ of \eqref{generalDiffusion}.  	\\
	
	More precisely, we consider a sequence $X^n=(X^n_t)_{t \leq T}$ of  continuous weak solutions to the stochastic Volterra equation \eqref{generalDiffusion} with a kernel $K^n \in \mathbb L^2([0,T], \R^{d \times d})$ admitting a resolvent of the first kind, on some filtered probability space $(\Omega^n, {\cal F}^n, \mathbb F^n ,\mathbb P^n )$,  
	\begin{equation*}
	X_t^n = g^n(t) + \int_0^t K^n(t-s) b(X_s^n) ds + \int_0^t K^n(t-s) \sigma(X_s^n) dW^n_s, \quad t \in [0,T],
	\end{equation*}
	 with $g^n: [0,T] \mapsto \R^d$ and $K^n$ satisfying \eqref{regularity_eq} for every $n \geq 1$. The stability of \eqref{generalDiffusion} means the convergence in law of the family of solutions $(X^n)_{n \geq 1}$ to a limiting process $X$ which is a solution to \eqref{generalDiffusion}, when $(K^n, g^n)$ is close to $(K,g)$ as $n$ goes large.\\
	 
	 This convergence is established by verifying first the Kolmogorov tightness criterion for the sequence $(X^n)_{n \geq 1}$. It is  obtained when $g^n$ and $K^n$ satisfy \eqref{regularity_eq} uniformly in $n$ in the following sense.
	\begin{assumption} \label{regularity_n} There exists positive constants $\gamma$ and $C$ such that
	$$ \sup_{n \geq 1} \left( |g^n(t+h) - g^n(t)|^2 + \int_0^h |K^n(s)|^2 ds + \int_0^{T - h} |K^n(h+s) - K^n(s)|^2 ds  \right) \leq C h^{2\gamma}, $$
	 for any $t, h \geq 0$ with $t + h \leq T$,
	\end{assumption}
	The following result, whose proof is postponed to Section \ref{stability_sec_proof} below, states the convergence of $(X^n)_{n \geq 1}$ to a solution of \eqref{generalDiffusion}.
	
	\begin{theorem} \label{stabilityTheorem} Assume that
	$$ \int_0^T | K(s) - K^n(s) |^2 ds  \longrightarrow 0 , \quad  g_n(t) \longrightarrow g(t), $$
	for any $t \in [0,T]$ as $n$ goes to infinity. Then, under Assumption~\ref{regularity_n}, the sequence $(X^n)_{n \geq 1}$ is tight for the uniform topology and any point limit $X$ is a solution of the stochastic Volterra equation \eqref{generalDiffusion}.
	\end{theorem}

\section{The particular case of the  rough Heston model}\label{rHestonSec}
The rough Heston model introduced in \cite{euch2016characteristic,euch2017perfect} is a particular case of the class of rough volatility models of Definition \ref{roughVolModel}, with $\sigma(x) = \nu \sqrt{x}$ for some positive parameter $\nu$, that is
	\begin{align*} 
	&\quad \quad\quad\quad\quad dS_t = S_t \sqrt{V_t} dW_t, \quad S_0 >0,\\
	&V_t = g(t) + \int_0^t K(t-s) \left(- \lambda V_s ds + \nu \sqrt{V_s}dB_s\right),
	\end{align*}
	 where $K(t) = \frac{t^{H - \frac12}}{\Gamma(H+1/2)}$ denotes the fractional kernel and $g$ is given by \eqref{casG}. Aside from reproducing accurately the historical and implied volatility, the rough Heston model displays a closed formula for the characteristic function of the log-price in terms of a solution  to a fractional Riccati equation allowing to fast pricing and calibration, see \cite{euch2017roughening}. More precisely, it is shown in \cite{papier2, euch2016characteristic, euch2017perfect} that
$$ L( t,z) = \E\big[ \exp \big( z \log(S_t / S_{0})\big) \big]$$
is given by 
\begin{equation} \label{formula1}
\exp\left( \int_0^t F(z, \psi(t-s,z)) g(s) ds\right), 
\end{equation}
where $\psi(\cdot,z)$ is the unique continuous solution of the fractional Riccati equation
\begin{equation} \label{fracRiccati}
\psi( t,z) = \int_0^t K(t-s) F(z, \psi(s,z))ds, \quad t \in [0,T],
\end{equation}
with $F(z,x) = \frac{1}{2}(z^2 - z) + (\rho \nu z - \lambda) x + \frac{\nu^2}{2} x^2 $ and  $z \in \C$ such that $\Re(z) \in [0,1]$. 
 Unlike the classical case $H = 1/2$, \eqref{fracRiccati} does not exhibit an explicit solution. However, it can be solved numerically through the Adam scheme developed in \cite{diethelm2002predictor,diethelm2004detailed,diethelm1998fracpece, euch2016characteristic} for instance. In this section, we show that the multi-factor approximation applied to the rough Heston model gives rise to another natural numerical scheme for solving the fractional Riccati equation. Furthermore, we will establish the convergence of this scheme with explicit errors. 

\subsection{Multi-factor scheme for the fractional Riccati equation}\label{multiSchemeSec}
We consider the multi-factor approximation $(S^n, V^n)$ of Definition \ref{multifactorModel} with $\sigma(x) = \nu \sqrt{x}$, where the number of factors $n$ is large, that is
$$ dS^n_t = S^n_t \sqrt{V^n_t} dW_t, \quad  V^n_t = g^n(t) + \sum_{i = 1}^n c_i^n V^{n,i}_t,$$
with
$$ dV^{n,i}_t = (-\gamma_i^n V^{n,i}_t - \lambda V_t^n) dt + \nu \sqrt{V_t^n} dB_t, \quad V^{n,i}_0 = 0, \quad S^n_0 = S_0 .$$
Recall that $g^n$ is given by \eqref{casG_n} and it converges pointwise to $g$ as $n$ goes large, see Lemma \ref{lemma1}.\\

We write the dynamics of $(S^n, V^n)$ in terms of a Volterra Heston model with the smoothed kernel $K^n$ given by \eqref{smoothedKernel} as follows
$$ dS^n_t = S^n_t \sqrt{V^n_t} dW_t, $$
$$ V^n_t = g^n(t) - \int_0^t K^n(t-s) \lambda V^n_s ds +  \int_0^t K^n(t-s) \nu \sqrt{V_s^n}  dB_s . $$
In \cite{papier2, ALP17}, the characteristic function formula of the log-price \eqref{formula1} is extended to the general class of Volterra Heston models. In particular, 
$$ L^n(t,z) = \E\big[ \exp \big( z \log(S_t^n / S_0)\big) \big]$$
is given by 
\begin{equation} \label{formula2}
\exp\left( \int_0^t F(z, \psi^n(t-s,z)) g^n(s) ds\right), 
\end{equation}
where $\psi^n( \cdot,z)$ is the unique continuous solution of the Riccati Volterra equation
\begin{equation} \label{volterraRiccati}
\psi^n(t,z) = \int_0^t K^n(t-s) F(z, \psi^n(s,z))ds, \quad t \in [0,T],
\end{equation}
for each  $z \in \C$ with $\Re(z) \in [0,1]$. \\

Thanks to the weak uniqueness of the rough Heston model, established in several works \cite{papier2, ALP17, mytnik2015uniqueness}, and to Theorem \ref{mainResult}, $(S^n, V^n)_{n \geq 1}$ converges in law for the uniform topology to $(S, V)$ when $n$ tends to infinity. In particular, $L^n(t,z)$ converges pointwise to $L(t,z)$. Therefore, we expect $\psi^n(\cdot,z)$ to be close to the solution of the fractional Riccati equation \eqref{fracRiccati}. This is the object of the next theorem, whose proof is reported to Section \ref{convergenceRicattiProof} below.
\begin{theorem} \label{convergenceRiccati} There exists a positive constant $C$ such that, for any $a \in [0,1]$, $b \in \R$ and $n \geq 1$,
$$ \sup_{t \in [0,T]} |\psi^n( t,a+ib) - \psi(t,a+ib)| \leq C (1 + b^4) \int_0^T |K^n(s) - K(s) |ds, $$
where $\psi(\cdot,a+ib)$  (resp.~$\psi^n(\cdot,a+ib)$) denotes the unique continuous solution of the Riccati Volterra equation \eqref{fracRiccati} (resp.~\eqref{volterraRiccati}).
\end{theorem}

Relying on the $L^1$-convergence of $(K^n)_{n \geq 1}$  to $K$ under Assumption \ref{factorsAssump}, see Proposition \ref{convergenceK}, we have the uniform convergence of $(\psi^n(\cdot,z))_{n \geq 1}$ to $\psi(\cdot,z)$ on $[0,T]$. Hence, Theorem \ref{convergenceRiccati} suggests a new numerical method for the computation of the fractional Riccati solution \eqref{fracRiccati} where an explicit error is given. Indeed, set 
$$ \psi^{n,i}(t,z) = \int_0^t e^{-\gamma_i^n (t-s)}  F(z, \psi^n(s,z)) ds, \quad  i \in \{1, \dots, n\} . $$
Then,
$$ \psi^n(t,z) = \sum_{i = 1}^n c_i^n \psi^{n,i}(t,z), $$
and $(\psi^{n,i}(\cdot,z))_{1 \leq i \leq n}$ solves the following $n$-dimensional system of ordinary Riccati equations
\begin{equation} \label{dimRiccati}
\partial_t \psi^{n,i}(t,z) = - \gamma_i^n \psi^{n,i}(t,z) + F(z,\psi^n(t,z)), \quad \psi^{n,i}(0,z)=0, \quad   i \in \{1, \dots, n\}.
\end{equation}
Hence, \eqref{dimRiccati} can be solved numerically by usual finite difference methods leading to $\psi^n(\cdot,z)$ as an approximation of the fractional Riccati solution. 

\subsection{Numerical illustrations}\label{illustrationSec}
In this section, we consider a rough Heston model with the following parameters
\begin{equation*}
\lambda = 0.3, \quad \rho = -0.7, \quad \nu = 0.3, \quad H = 0.1, \quad V_0 = 0.02 , \quad \theta \equiv 0.02. 
\end{equation*}
We discuss the accuracy of the multi-factor approximation sequence $(S^n, V^n)_{n \geq 1}$ as well as the corresponding Riccati Volterra solution $(\psi^n(\cdot,z))_{n \geq 1}$, for different choices of the weights $(c_i^n)_{1 \leq i \leq n}$ and mean reversions $(\gamma_i^n)_{1 \leq i \leq n}$. This is achieved by first computing, for different number of factors $n$, the implied volatility $\sigma^n(k, T)$ of maturity $T$ and log-moneyness $k$ by a Fourier inversion of the   characteristic function formula \eqref{formula2}, see \cite{carr1999option, lewis2001simple} for instance. In a second step, we compare $\sigma^n(k,T)$  to the implied volatility $\sigma(k, T)$ of the rough Heston model. We also compare the Riccati Volterra solution $\psi^n(T, z)$ to the fractional one $\psi(T,z)$. \\

\noindent Note that the Riccati Volterra solution $\psi^n(\cdot,z)$ is computed by solving numerically the $n$-dimensional Riccati equation \eqref{dimRiccati} with a classical finite difference scheme. The complexity of such scheme is $O(n \times n_{\Delta t})$, where $n_{\Delta t}$ is the number of time steps applied for the scheme, while the complexity of the Adam scheme used for the computation of $\psi(\cdot,z)$ is $O(n_{\Delta t}^2)$. In the following numerical illustrations, we fix $n_{\Delta t} = 200$. \\

In order to guarantee the convergence, the weights and mean reversions have to satisfy Assumption \ref{factorsAssump} and in particular they should be of the form \eqref{E:cgamma} in terms of auxiliary mean reversions $(\eta_i^n)_{0\leq i\leq n}$ satisfying \eqref{condFacteurs}. For instance, one can fix
\begin{equation} \label{factorChoice1} 
\eta_i^n = i \pi_n  , \quad i \in \{0, \dots, n\} ,
\end{equation}
where $\pi_n$ is defined by \eqref{pasOptimal}, as previously done in Section \ref{choiceK}. For this particular choice, Figure \ref{figL2_ric} shows a decrease of the relative error $\left| \frac{\psi^n(T, ib) - \psi(T,ib)}{\psi(T,ib)}\right|$ towards zero for different values of $b$. 
\begin{center}
\includegraphics[scale=0.55]{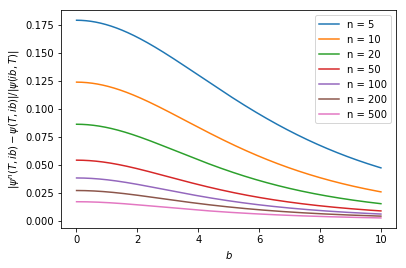}
\captionof{figure}{The relative error $\left| \frac{\psi^n(T, ib) - \psi(T,ib)}{\psi(T,ib)}\right|$ as a function of $b$ under \eqref{factorChoice1} and for different numbers of factors $n$ with $T = 1$.}
\label{figL2_ric}
\end{center}
We also observe in the Figure \ref{figL2} below that the implied volatility $\sigma^n(k,T)$ of the multi-factor approximation is close to $\sigma(k,T)$ for a number of factors $n \geq 20$. Notice that the approximation is more accurate around the money.
\begin{center}
\includegraphics[scale=0.55]{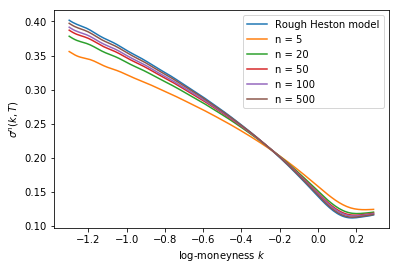}
\captionof{figure}{Implied volatility $\sigma^n(k, T)$ as a function of the log-moneyness $k$ under \eqref{factorChoice1} and for different numbers of factors $n$ with $T = 1$.}
\label{figL2}
\end{center}

In order to obtain a more accurate convergence, we can minimize the upper bound $f_n^{(2)}((\eta_i^n)_{0 \leq i \leq n})$ of $ \|K^n - K\|_{2, T} $ defined in \eqref{f_n_2}. Hence, we choose $(\eta_i^n)_{0 \leq i \leq n}$ to be a solution of the constrained minimization problem
\begin{equation} \label{factorChoice2}
 \inf_{(\eta_i^n)_i \in {\cal E}_n} f_n^{(2)}((\eta_i^n)_{0 \leq i \leq n}),
\end{equation}
where  ${\cal E}_n = \{(\eta_i^n)_{0 \leq i \leq n} ; \quad 0 = \eta_0^n < \eta_1^n < ... < \eta_n^n \}$.

\begin{center}
\includegraphics[scale=0.55]{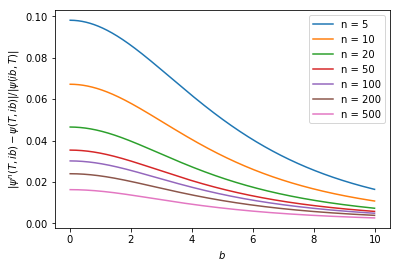}
\captionof{figure}{The relative error $\left| \frac{\psi^n(T,ib) - \psi(T,ib)}{\psi(T,ib)}\right|$ as a function of $b$ under \eqref{factorChoice2} and for different numbers of factors $n$ with $T = 1$.}
\label{figL2_ric_min}
\end{center}
We notice from Figure \ref{figL2_ric_min}, that the relative error $| \frac{\psi^n(T,ib) - \psi(T,ib)}{\psi(T,ib)}|$ is smaller under the choice of factors \eqref{factorChoice2}. Indeed the Volterra approximation $\psi^n(T,ib)$ is now closer to the fractional Riccati solution $\psi(T, ib)$  especially for small number of factors. However, when $n$ is large, the accuracy of the approximation seems to be close to the one under \eqref{factorChoice1}. For instance when $n = 500$, the relative error is around $1\%$ under both \eqref{factorChoice1} and \eqref{factorChoice2}.

\begin{center}
\includegraphics[scale=0.55]{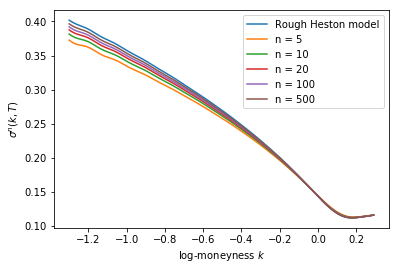}
\captionof{figure}{Implied volatility $\sigma^n(k, T)$ as a function of the log-moneyness $k$ under \eqref{factorChoice2} and for different numbers of factors $n$ with $T = 1$.}
\label{figL2_min}
\end{center}
In the same way, we observe in Figure \ref{figL2_min} that the accuracy of the implied volatility approximation $\sigma^n(k,T)$ is more satisfactory under \eqref{factorChoice2} especially for a small number of factors.\\

\noindent Theorem \ref{convergenceRiccati} states that the convergence of $\psi^n(\cdot,z)$ depends actually on the $\mathbb L^1([0,T], \R)$-error between $K^n$ and $K$. Similarly to the computations of Section \ref{choiceK}, we may show that,
$$ \int_0^T |K^n(s) - K(s)| ds \leq f_n^{(1)}((\eta_i^n)_{0 \leq i \leq n}),$$
where
$$f_n^{(1)}((\eta_i^n)_{0 \leq i \leq n}) = \frac{T^{3}}{6} \sum_{i = 1}^n  \int_{\eta_{i-1}^n}^{\eta_i^n}(\gamma - \gamma_i^n)^2 \mu(d\gamma) + \frac{1}{\Gamma(H+3/2) \Gamma(1/2-H) } (\eta^n_n)^{-H - \frac12}. $$
This leads to choosing $(\eta_i^n)_{0 \leq i \leq n}$ as a solution of the constrained minimization problem
\begin{equation} \label{factorChoice3}
 \inf_{(\eta_i^n)_i \in {\cal E}_n} f_n^{(1)}((\eta_i^n)_{0 \leq i \leq n}).
\end{equation}
It is easy to show that such auxiliary mean-reversions $(\eta_i^n)_{0 \leq i \leq n}$ satisfy \eqref{condFacteurs} and thus Assumption \ref{factorsAssump} is met. 
\begin{center}
\includegraphics[scale=0.55]{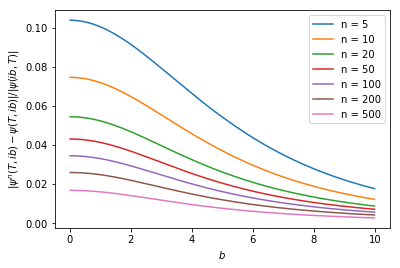}
\captionof{figure}{The relative error $\left| \frac{\psi^n(T,ib) - \psi(T,ib)}{\psi(T,ib)}\right|$ as a function of $b$ under \eqref{factorChoice3} and for different numbers of factors $n$ with $T = 1$.}
\label{figL1_ric_min}
\end{center}
\begin{center}
\includegraphics[scale=0.55]{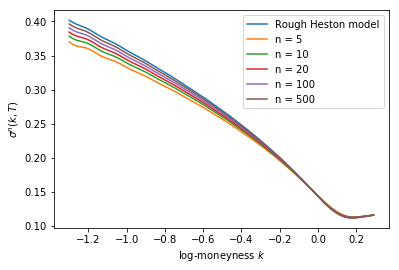}
\captionof{figure}{Implied volatility $\sigma^n(k, T)$ as a function of the log-moneyness $k$ under \eqref{factorChoice3} and for different numbers of factors $n$ with $T = 1$.}
\label{figL1_min}
\end{center}

Figures \ref{figL1_ric_min} and \ref{figL1_min} exhibit similar results as the ones in Figures \ref{figL2_ric_min} and \ref{figL2_min} corresponding to the choice of factors \eqref{factorChoice2}. In fact, we notice in practice that the solution of the minimization problem \eqref{factorChoice2} is close to the one in \eqref{factorChoice3}.

\subsection{Upper bound for call prices error}
Using a Fourier transform method, we can also provide an error between the price of the call $C^n(k,T) = \E[(S^n_T - S_0 e^k)_+]$ in the multi-factor model and the price of the same call $C(k, T) = \E[(S_T-S_0 e^k)_+]$ in the rough Heston model. However, for technical reasons, this bound is obtained for a modification of the multi-factor approximation $(S^n, V^n)_{n \geq 1}$ of Definition \ref{multifactorModel} where the function $g^n$ given initially by \eqref{casG_n} is updated into
\begin{equation} \label{casGn_2}
g^n(t) = \int_0^t K^n(t-s) \big(V_0 \frac{s^{-H-\frac12}}{\Gamma(1/2-H)} + \theta(s)\big) ds,
\end{equation}
where $K^n$ is the smoothed approximation \eqref{smoothedKernel} of the fractional kernel. Note that the strong existence and uniqueness of $V^n$ is still directly obtained from Proposition \ref{T:Yamada} and its non-negativity from Theorem \ref{existencePositive} together with Remarks \ref{positiveFrac} and \ref{postiiveG_rem} in the Appendix\footnote{Note that Theorem \ref{existencePositive} is used here for the smoothed kernel $K^n$, $b(x) = - \lambda x$ and $g^n$ defined by \eqref{casGn_2}.}. Although for $g^n$ satisfying \eqref{casGn_2}, $(V^n)_{n \geq 1}$ can not be tight\footnote{In fact, $V^n_0 = 0$ while $V_0$ may be positive.}, the corresponding spot price $(S^n)_{n \geq 1}$ converges as shown in the following proposition.
\begin{proposition} \label{ConvergenceHeston} Let $(S^n, V^n)_{n \geq 1}$ be a sequence of multi-factor Heston models as in Definition \ref{multifactorModel} with $\sigma(x) = \nu \sqrt{x}$ and $g^n$ given by \eqref{casGn_2}. Then, under Assumption \ref{factorsAssump}, $(S^n, \int_0^\cdot V^n_s ds)_{n \geq 1}$ converges in law for the uniform topology to $(S, \int_0^\cdot V_s ds)$, where $(S, V)$ is a rough Heston model as in Definition \ref{roughVolModel} with $\sigma(x) = \nu \sqrt{x}$. 
\end{proposition} 

Note that the characteristic function \eqref{formula2} still holds. Using Theorem \ref{convergenceRiccati} together with a Fourier transform method, we obtain an explicit error for the call prices. We refer to Section \ref{callErrorProof} below for the proof.

\begin{proposition} \label{callError} Let $C(k,T)$ be the  price of the call in the rough Heston model with maturity $T>0$ and log-moneyness $k \in \R$. We denote by $C^n(k, T)$ the price of the call in the multi-factor Heston model of Definition \ref{multifactorModel} such that $g^n$ is given by \eqref{casGn_2}. If $|\rho| < 1$, then there exists a positive constant $c > 0$ such that
$$ |C(k,T) - C^n(k,T) | \leq c \int_0^T |K(s) - K^n(s)| ds, \quad n \geq 1. $$
\end{proposition}

\section{Proofs}\label{proofSec} In this section, we use the convolution notations together with the resolvent definitions of Appendix \ref{ConvolComputations}. We denote by  $c$  any positive constant independent of the variables $t, h$ and $n$ and that may vary from line to line. For any $h \in \R$, we will use the notation $\Delta_h$  to denote the semigroup operator of right shifts defined by $\Delta_h f : t \mapsto f(h+t)$ for any function $f$.\\

We first prove Theorem \ref{stabilityTheorem}, which is the building block of Theorem \ref{mainResult}. Then, we turn to the proofs of the results contained in Section \ref{rHestonSec}, which concern the particular case of the rough Heston model. 
	 
\subsection{Proof of Theorem \ref{stabilityTheorem}} \label{stability_sec_proof}

	\paragraph{Tightness of $(X^n)_{n \geq 1}$ :} We first show that, for any $ p \geq 2$,
	\begin{equation} \label{uniformBoundedness}
	 \sup_{n \geq 1}~ \sup_{t \leq T}~  \E[|X^n_t|^p]  < \infty.
	\end{equation}
	Thanks to Proposition \ref{weakExistence}, we already have
	 \begin{equation} \label{EX_n}
	 \sup_{t \leq T}~\E[|X^n_t|^p] < \infty.
	\end{equation}
         Using the linear growth of $(b,\sigma)$ and \eqref{EX_n} together with Jensen and BDG inequalities, we get 
	$$ \E[|X^n_t|^p] \leq  c \left(\sup_{t\leq T} |g^n(t)|^p + \left( \int_0^{T} |K^n(s)|^2 ds \right)^{\frac{p}2 - 1}  \int_0^t |K^n(t-s)|^2 (1+ \E[|X_s^n|^p]) ds )\right).$$
	Relying on Assumption \ref{regularity_n} and the convergence of $(g^n(0), \int_0^T |K^n(s)|^2 ds)_{n \geq 1}$, $\sup_{t\leq T} |g^n(t)|^p$ and $\int_0^T |K^n(s)|^2 ds$ are uniformly bounded in $n$. This leads to
	$$ \E[|X^n_t|^p] \leq  c \left(1 + \int_0^{t} |K^n(t-s)|^2 \E[|{X_s^n}|^p] ds )\right).$$
	By the Gr\"onwall type inequality in Lemma \ref{gronwall} in the Appendix, we deduce that
	$$ \E[|X^n_t|^p] \leq  c \left(1 +  \int_0^{t} E_c^n(s) ds )\right) \leq c {\left(1 +  \int_0^{T} E_c^n(s) ds )\right)} ,$$
	where $E_c^n \in \mathbb L^1([0,T], \R) $ is the canonical resolvent of $|K^n|^2$  with parameter $c$, defined in Appendix \ref{resolvent2}, and the last inequality follows from the fact that $\int_0^\cdot E_c^n(s)ds $ is non-decreasing by Corollary \ref{resolvent_positive}. The convergence of $|K^n|^2$ to $|K|^2$ in $\mathbb L^1([0,T], \R)$ implies the convergence of $E^n_c$ to the canonical resolvent of $|K|^2$ with parameter $c$  in $\mathbb L^1([0,T], \R)$, see \cite[Theorem 2.3.1]{GLS90}. Thus, $ \int_0^T E_c^n(s) ds $ is uniformly bounded in $n$, yielding \eqref{uniformBoundedness}. \\ 
	
	We now show that $(X^n)_{n \geq 1}$ exhibits the Kolmogorov tightness criterion. In fact, using again the linear growth of ($b$, $\sigma$) and \eqref{uniformBoundedness} together with Jensen and BDG inequalities, we obtain, for any $p\geq2$ and $t, h \geq 0$ such that $t + h \leq T$, 
	$$ \E[|X_{t+h}^n - X_t^n|^p] \leq c \Big(|g^n(t+h) -g^n(t)|^p + \big(\int_0^{T-h}   \! \!\! \! \!|K^n(h+s) - K^n(s)|^2 ds \big)^{p/2} + \big(  \int_0^h  \! \! |K^n(s)|^2 ds\big)^{p/2} \Big). $$
	Hence, Assumption \ref{regularity_n} leads to
	$$ \E[|X_{t+h}^n - X_t^n|^p] \leq c h^{p\gamma}, $$
	and therefore to the tightness of $(X^n)_{n \geq 1}$ for the uniform topology. \\
	
	\paragraph{Convergence of $(X^n)_{n \geq 1}$ :} Let $M^n_t = \int_0^t \sigma(X^n_s) dW_s^n$. As $\langle M^n \rangle_t =  \int_0^t \sigma\sigma^*(X^n_s) ds$, $(\langle M^n \rangle)_{n \geq 1}$ is tight and consequently we get the tightness of $(M^n)_{n \geq 1}$ from \cite[Theorem VI-4.13]{jacod2013limit}. Let $(X, M) = (X_t, M_t)_{t \leq T}$ be a possible limit point of  $(X^n, M^n)_{n \geq 1}$. Thanks to \cite[Theorem VI-6.26]{jacod2013limit}, $M$ is a local martingale and necessarily
	$$ \langle M \rangle_t = \int_0^t \sigma \sigma^*(X_s) ds, \quad t \in [0,T]. $$	
	Moreover, setting $Y^n_t = \int_0^t b(X^n_s) ds + M^n_t$, the assoicativity property \eqref{stochasticFubini} in the Appendix yields
	\begin{equation} \label{equality_n} 
(L*X^n)_t = (L*g^n)(t) + \left(L * \big((K^n - K) * dY^n\big)\right)_t + Y^n_t ,
	\end{equation}
	where $L$ is the resolvent of the first kind of $K$ defined in Appendix \ref{resolvent1}.
	By the Skorokhod representation theorem, we construct a probability space supporting a sequence of copies of $(X^n, M^n)_{ n \geq 1}$ that converges uniformly on $[0,T]$, along a subsequence,  to a copy of $(X, M)$ almost surely, as $n$ goes to infinity. We maintain the same notations for these copies. Hence, we have
	$$ \sup_{t \in [0,T]} |X^n_t - X_t| \rightarrow 0, \quad \sup_{t \in [0,T]} |M^n_t - M_t| \rightarrow 0,  $$
	almost surely, as $n$ goes to infinity. Relying on the continuity and linear growth of $b$ together with the dominated convergence theorem, it is easy to obtain for any $t \in [0, T]$
	$$ (L*X^n)_t \rightarrow (L*X)_t, \quad \int_0^t b(X^n_s) ds \rightarrow \int_0^t b(X_s) ds ,$$
	almost surely as $n$ goes to infinity. Moreover for each $t \in [0,T]$
	$$ {(L*g^n)(t) \rightarrow (L*g)(t)}, $$
	by the uniform boundedness of $g^n$ in $n$ and $t$ and the dominated convergence theorem.  Finally thanks to the Jensen inequality,
	$$  {\E [| \left(L *( (K^n - K) * dY^n)\right)_t|^2] \leq c \sup_{t \leq T} \E[|\left((K^n-K)*dY^n\right)_t|^2]} .$$
	From \eqref{uniformBoundedness} and the linear growth of $(b ,\sigma)$, we deduce
	$$ { \sup_{t \leq T} \E[|\left((K^n-K)*dY^n\right)_t|^2]} \leq c \int_0^T |K^n(s)-K(s)|^2 ds, $$	
	which goes to zero when $n$ is large.
	Consequently, we send $n$ to infinity in \eqref{equality_n} and obtain the following almost surely equality, for each $t \in [0,T]$,
	\begin{equation} \label{equality}  
	{(L*X)_t = (L*g)(t)} + \int_0^t b(X_s)ds + M_t .
	\end{equation}
	Recall also that $ \langle M \rangle = \int_0^\cdot \sigma \sigma^*(X_s) ds $. Hence, by \cite[Theorem V-3.9]{revuz2013continuous}, there exists a $m$-dimensional Brownian motion $W$ such that
	$$ M_t = \int_0^t \sigma(X_s) dW_s , \quad t \in [0,T] .$$ 
	The processes in \eqref{equality} being continuous, we deduce that, almost surely,
	 $$ {(L*X)_t = (L*g)(t)} + \int_0^t b(X_s)ds + \int_0^t \sigma(X_s) dW_s, \quad t \in [0,T].$$
	 We convolve by $K$ and use the associativity property \eqref{stochasticFubini} in the Appendix to get that, almost surely,
	 $$  \int_0^t X_s ds  = \int_0^t g(s) ds + \int_0^t \left(\int_0^s K(s-u) (b(X_u)du + \sigma(X_u) dW_u) \right)ds, \quad t \in [0,T]. $$
	 Finally it is easy to see that the processes above are differentiable and we conclude that $X$ is  solution of the stochastic Volterra equation \eqref{generalDiffusion},  by taking the derivative.
	 
\subsection{Proof of Theorem \ref{mainResult}}\label{mainResultProof}
	 Theorem \ref{mainResult} is easily obtained once we prove the tightness of $(V^n)_{n \geq 1}$ for the uniform topology and that any limit point $V$ is solution of the fractional stochastic integral equation \eqref{roughVol}. This is a direct consequence of Theorem \ref{stabilityTheorem}, by setting $d=m=1$, $g$ and $g^n$ respectively as in \eqref{casG} and \eqref{casG_n}, $b(x) = -\lambda x$, $K$ being the fractional kernel and $K^n(t) = \sum_{i=1}^n c_i^n e^{-\gamma_i^n t}$ its smoothed approximation. Under Assumption \ref{factorsAssump}, $(K^n)_{n \geq 1}$ converges in $\mathbb L^2([0,T], \R)$ to the fractional kernel, see Proposition \ref{convergenceK}. Hence, it is left to show the pointwise convergence of $(g^n)_{n \geq 1}$ to $g$ on $[0,T]$ and that $(K^n, g^n)_{n \geq 1}$ satisfies Assumption \ref{regularity_n}. 
	 \begin{lemma}[Convergence of $g^n$] \label{lemma1} Define $g^n : [0,T] \mapsto \R$ and $g : [0,T] \mapsto \R$  respectively by \eqref{casG} and \eqref{casG_n} such that $\theta : [0,T] \mapsto \R$ satisfies \eqref{cond1}. Under assumption \eqref{factorsAssump}, we have for any $t \in [0,T]$
	 $$ g^n(t) \rightarrow g(t) , $$
	 as $n$ tends to infinity.
	 \end{lemma}
	\begin{proof}
	Because $\theta$ satisfies \eqref{cond1}, it is enough to show that for each $t \in [0,T]$
	\begin{equation} \label{term_inter}
	\int_0^t (t-s)^{-\frac12 - \varepsilon} |K^n(s) - K(s)| ds 
	\end{equation}
	converges to zero as $n$ goes large, for some $\varepsilon > 0$ and $K^n$ given by \eqref{smoothedKernel}. Using the  representation of $K$ as the Laplace transform of $\mu$ as in \eqref{Laplace}, we obtain that \eqref{term_inter} is bounded by
	\begin{equation} \label{term_inter2}
	 \int_0^t (t-s)^{-\frac12 - \varepsilon} \int_{\eta_n^n}^\infty e^{-\gamma s} \mu(d\gamma) ds  + \sum_{i=1}^n \int_0^t (t-s)^{-\frac12 - \varepsilon} |c_i^n e^{- \gamma_i^n s} - \int_{\eta_{i-1}^n}^{\eta_{i}^n} e^{-\gamma s} \mu(d\gamma)| ds .
	 \end{equation}
	 The first term in \eqref{term_inter2} converges to zero for large $n$ by the dominated convergence theorem because $\eta_n^n$ tends to infinity, see Assumption \ref{factorsAssump}. Using the Taylor-Lagrange inequality \eqref{tayllor}, the second term in \eqref{term_inter2} is dominated by
	 \begin{equation*}
	 \frac12 \int_0^t (t-s)^{-\frac12-\varepsilon} s^{2} ds \sum_{i=1}^n \int_{\eta_{i-1}^n}^{\eta_i^n} (\gamma- \gamma_i^n)^2 \mu(d\gamma), 
	 \end{equation*}
	 which goes to zero thanks to Assumption \ref{factorsAssump}.
	\end{proof}

	 \begin{lemma}[$K^n$ satisfying Assumption \ref{regularity_n}]\label{uniform_reg} Under Assumption \ref{factorsAssump}, there exists $C > 0$ such that, for any $t, h \geq 0$ with $t+h \leq T$,
	$$ \sup_{n \geq 1} ~ \left(\int_0^{T-h} | K^n(h+s) - K^n(s) |^2 ds + \int_0^h |K^n(s)|^2 ds \right) \leq C h^{2H}, $$
	where $K^n$ is defined by \eqref{smoothedKernel}.
	 \end{lemma}
	 
	 \begin{proof} We start by proving that for any $t, h \geq 0$ with $t+h \leq T$
	 \begin{equation} \label{reg1}
	  \int_0^h |K^n(s)|^2 ds \leq c h^{2H}.
	 \end{equation}
	 In fact we know that this inequality is satisfied for $K(t) = \frac{t^{H-\frac12}}{\Gamma(H+1/2)}$. Thus it is enough to prove 
	 $$  \|K^n - K\|_{2,h}  \leq c h^{H}, $$
	 where $\|\cdot\|_{2,h}$ stands for the usual $\mathbb L^2([0,h], \R)$ norm.  Relying on the Laplace transform representation of $K$ given by \eqref{Laplace}, we obtain
	 	 $$ \| K^n - K \|_{2,h}  \leq \int_{\eta_n^n}^\infty {\| e^{-\gamma (\cdot)} \|_{2,h}} \mu(d\gamma) + \sum_{i = 1}^n J_{i,h}^n, $$
	where $J_{i,h}^n = {\| c_i^n e^{- \gamma_i^n (\cdot)} - \int_{\eta_{i-1}^n}^{\eta_i^n} e^{-\gamma (\cdot)} \mu(d\gamma) \|_{2, h}} $. We start by bounding the first term,
	\begin{align*}  
	\int_{\eta_n^n}^\infty {\| e^{-\gamma (\cdot)} \|_{2,h}} \mu(d\gamma) &\leq  \int_{0}^\infty {\sqrt{\frac{1-e^{-2\gamma h}}{2\gamma}}} \mu(d\gamma)\\
	 &={ \frac{h^{H}}{\Gamma(H+1/2) \Gamma(1/2-H)\sqrt{2}} \int_{0}^\infty \sqrt{\frac{1-e^{-2\gamma}}{\gamma}} \gamma^{-H-\frac12}d\gamma }.  
	\end{align*}
	As in Section \ref{choiceK}, we use the Taylor-Lagrange inequality \eqref{tayllor} to get  
	$$ \sum_{i = 1}^{n} J_{i,h}^n \leq \frac1{2\sqrt5} h^{\frac52} \sum_{i=1}^n \int_{\eta_{i-1}^n}^{\eta_i^n} (\gamma- \gamma_i^n)^2 \mu(d\gamma) .$$
	Using the boundedness of $\big( \sum_{i=1}^n \int_{\eta_{i-1}^n}^{\eta_i^n} (\gamma- \gamma_i^n)^2 \mu(d\gamma) \big)_{n \geq 1}$ from Assumption \ref{factorsAssump}, we deduce \eqref{reg1}. 	We now prove
	\begin{equation} \label{reg2}
	  \int_0^{T-h} |K^n(h+s) -K^n(s) |^2 ds \leq c h^{2H}.
	 \end{equation}
	In the same way, it is enough to show 
	$$ \| (\Delta_h K^n - \Delta_h K) - (K^n - K) \|_{2,T-h} \leq c h^H, $$
	Similarly to the previous computations, we get
	 $$ \| (\Delta_h K^n - \Delta_h K) - (K^n - K) \|_{2,T-h} \leq \int_{\eta_n^n}^\infty {\| e^{-\gamma (\cdot)} - e^{-\gamma (h+\cdot)} \|_{2,T-h}} \mu(d\gamma) + \sum_{i = 1}^n {\widetilde J^n_{i,h}}, $$
	 with ${\widetilde J^n_{i,h} = \| c_i^n (e^{- \gamma_i^n (\cdot)} - e^{- \gamma_i^n (h+\cdot)}) - \int_{\eta_{i-1}^n}^{\eta_i^n} (e^{-\gamma (\cdot)}- e^{-\gamma( h+\cdot)})\mu(d\gamma) \|_{2, T-h}} $.  Notice that 
	\begin{align*}
	\int_{\eta_n^n}^\infty \| e^{-\gamma (\cdot)} - e^{-\gamma (h+\cdot)} \|_{2,T-h} \mu(d\gamma) &= \int_{\eta_n^n}^\infty (1-e^{-\gamma h})  \sqrt{\frac{1-e^{-2\gamma(T-h)}}{2 \gamma}} \mu(d\gamma)\\
		 &\leq c \int_0^\infty (1-e^{-\gamma h})   \gamma^{-H-1} d\gamma \leq c h^H.
	\end{align*}  
	Moreover, fix $h,t>0$ and set $\chi(\gamma) = e^{-\gamma t} - e^{-\gamma(t+h)}$. The second derivative reads
	 \begin{equation}  \label{ineq_taylor2}
	 \chi''(\gamma) = h \big(t^{2} \gamma e^{- \gamma t} \frac{1- e^{-\gamma h}}{\gamma h}  - h e^{-\gamma(t+h)} - 2 t e^{-\gamma (t+h)} \big), \quad \gamma>0.
	 \end{equation}
	 Because $x  \mapsto x e^{-x}$ and $x \mapsto \frac{1-e^{-x}}{x}$ are bounded functions on $(0, \infty)$, there exists $C > 0$ independent of $t,h \in [0,T]$ such that
	 $$ |\chi''(\gamma)| \leq C h , \quad \gamma > 0.$$
	 The Taylor-Lagrange formula, up to the second order, leads to
	$$ | c_i^n (e^{- \gamma_i^n t} - e^{- \gamma_i^n (t+h)}) - \int_{\eta_{i-1}^n}^{\eta_i^n} (e^{-\gamma t}- e^{-\gamma( t+h)})\mu(d\gamma)| \leq   \frac{C}{2} h \int_{\eta_{i-1}^n}^{\eta_i^n} (\gamma - \gamma_i^n)^2 \mu(d\gamma) . $$
	Thus
	$$ \sum_{i = 1}^n  {\widetilde J^n_{i,h}} \leq \frac{C}{2} h \sum_{i=1}^n \int_{\eta_{i-1}^n}^{\eta_i^n} (\gamma - \gamma_i^n)^2 \mu(d\gamma) .$$
	Finally, \eqref{reg2} follows from the boundedness of $\big( \sum_{i=1}^n \int_{\eta_{i-1}^n}^{\eta_i^n} (\gamma- \gamma_i^n)^2 \mu(d\gamma) \big)_{n \geq 1}$ due to Assumption \ref{factorsAssump}.
	 \end{proof}
\begin{lemma}[$g^n$ satisfying Assumption \ref{regularity_n}] Define $g^n : [0,T] \mapsto \R$ by \eqref{casG_n} such that $\theta : [0,T] \mapsto \R$ satisfies \eqref{cond1}. Under Assumption \ref{factorsAssump}, for each $\varepsilon>0$, there exists $C_\varepsilon > 0$ such that for any $t, h \geq 0$ with $t + h \leq T$
	 $$ \sup_{n \geq 1}~ |g^n(t) - g^n(t+h)| \leq  C_\varepsilon h^{H-\varepsilon} . $$
	 \end{lemma}
\begin{proof} 
Because $\theta$ satisfies \eqref{cond1}, it is enough to prove that, for each fixed $\varepsilon > 0$, there exists $C>0$ such that 
	\begin{equation} \label{term01}
	\sup_{n \geq 1} \int_0^h (h-s)^{-\frac12 - \varepsilon} |K^n(s)| ds \leq C h^{H-\varepsilon},
	\end{equation}
	and
	\begin{equation} \label{term02}
	\sup_{n \geq 1} \int_0^t (t-s)^{-\frac12 - \varepsilon} |K^n(s) - K^n(h+s)| ds \leq C h^{H-\varepsilon},
	\end{equation}
	for any $t, h \geq 0$ with $t+ h \leq T$. \eqref{term01} being satisfied for the fractional kernel, it is enough to establish 
		$$ \int_0^h (h-s)^{-\frac12 - \varepsilon} |K^n(s) - K(s)| ds \leq c h^{H-\varepsilon}. $$
	In the proof of Lemma \ref{lemma1}, it is shown that 
	$$ \int_0^h (h-s)^{-\frac12 - \varepsilon} |K^n(s) - K(s)| ds $$
	is bounded by \eqref{term_inter2}, that is
	$$  \int_0^h (h-s)^{-\frac12 - \varepsilon} \int_{\eta_n^n}^\infty e^{-\gamma s} \mu(d\gamma) ds  + \sum_{i=1}^n \int_0^h (h-s)^{-\frac12 - \varepsilon} |c_i^n e^{- \gamma_i^n s} - \int_{\eta_{i-1}^n}^{\eta_{i}^n} e^{-\gamma s} \mu(d\gamma)| ds . $$
	The first term is dominated by
	$$ \int_0^h (h-s)^{-\frac12 - \varepsilon} \int_{0}^\infty e^{-\gamma s} \mu(d\gamma) ds = h^{H - \varepsilon} \frac{B(1/2- \varepsilon, H +1/2)}{B(1/2-H, H+1/2)} , $$
	where $B$ is the usual Beta function. Moreover thanks to \eqref{tayllor} and Assumption \ref{factorsAssump}, we get
	$$  \sum_{i=1}^n \int_0^h (h-s)^{-\frac12 - \varepsilon} |c_i^n e^{- \gamma_i^n s} - \int_{\eta_{i-1}^n}^{\eta_{i}^n} e^{-\gamma s} \mu(d\gamma)| ds \leq c h^{\frac52-\varepsilon}, $$
	yielding \eqref{term01}. Similarly, we obtain \eqref{term02} by showing that  
	$$ \int_0^t (t-s)^{-\frac12 - \varepsilon} \left|(K^n(s) - \Delta_hK^n(s)) - (K(s) - \Delta_hK(s)) \right| ds \leq c h^{H-\varepsilon}. $$
	By similar computations as previously and using \eqref{ineq_taylor2}, we get that
	$$ \int_0^t (t-s)^{-\frac12 - \varepsilon} \left|(K^n(s) - \Delta_hK^n(s)) - (K(s) - \Delta_hK(s)) \right| ds $$
	is dominated by
	$$c \left( \int_0^t (t-s)^{\frac12-\varepsilon} \int_{\eta_n^n}^\infty (1-e^{-\gamma h})e^{-\gamma s} \mu(d\gamma) ds + h \sum_{i=1}^n\int_{\eta_{i-1}^n}^{\eta_i^n} (\gamma - \gamma_i^n)^2 \mu(d\gamma) \right). $$
	The first term being bounded by
	$$  \int_0^t (t-s)^{\frac12-\varepsilon} \int_0^\infty (1-e^{-\gamma h})e^{-\gamma s} \mu(d\gamma) ds  = \int_0^t (t-s)^{\frac12-\varepsilon} (K(s) - K(h+s))ds \leq c h^{H - \varepsilon} , $$
	Assumption \ref{factorsAssump} leads to \eqref{term02}.
	
\end{proof}

\subsection{Proof of Theorem \ref{convergenceRiccati}} \label{convergenceRicattiProof}
\paragraph{Uniform boundedness :} We start by showing the uniform boundedness of the unique continuous solutions $(\psi^n(\cdot,a+ib))_{n \geq 1}$ of \eqref{volterraRiccati}.
\begin{proposition} \label{boundednessRiccati} 
For a fixed $T>0$, there exists $C>0$ such that
$$ \sup_{n \geq 1}\sup_{t \in [0,T]} |\psi^n(t,a+ib)| \leq C \left( 1 + b^2  \right) ,$$
for any $a \in [0,1]$ and $b \in \R$. \end{proposition}

\begin{proof} Let $z = a + ib$ and start by noticing that $\Re(\psi^n(\cdot,z))$ is non-positive because it solves the following linear Volterra equation with continuous coefficients
$$ \chi = K^n* \left( f + \left(\rho \nu \Re(z) - \lambda + \frac{\nu^2}2  \Re(\psi^n(\cdot,z))\right) \chi\right) ,$$
where 
$$ f = \frac12 \left( a^2-a- (1-\rho^2) b^2 \right) -  \frac{1}{2} (\rho b + \nu\psi^n(\cdot,z))^2 $$
is continuous non-positive, see Theorem \ref{T:positive_3}. In the same way $\Re(\psi(\cdot, z))$ is also non-positive. Moreover, observe that $\psi^n(\cdot,z)$ solves the following linear Volterra equation with continuous coefficients
$$ \chi = K^n* \left(  \frac12 ( z^2-z)  + ( \rho \nu z- \lambda + \frac{\nu^2}{2} \psi^n(\cdot,z) ) \chi \right) ,$$
and 
$$ \Re\left( \rho \nu z- \lambda + \frac{\nu^2}{2} \psi^n(\cdot,z) \right) \leq  \nu  - \lambda. $$
Therefore, Corollary \ref{corolC} leads to
$$ \sup_{t \in [0,T]}|\psi^n(t,z)| \leq \frac{1}{2} |z^2 - z| \int_0^T E^n_{\nu - \lambda}(s) ds, $$
where $E^n_{\nu - \lambda}$ denotes the canonical resolvent of $K^n$ with parameter $ \nu - \lambda$, see Appendix \ref{resolvent2}. This resolvent converges in $\mathbb L^1([0,T], \R)$ because $K^n$ converges in $\mathbb L^1([0,T], \R)$ to $K$, see \cite[Theorem 2.3.1]{GLS90}. Hence,  $(\int_0^T E^n_{ \nu - \lambda}(s) ds)_{n \geq 1}$ is bounded, which ends the proof.

\end{proof}

\paragraph{End of the proof of Theorem \ref{convergenceRiccati} :} Set $z = a + ib$ and recall that 
$$ \psi^n(\cdot,z) = K^n * F(z, \psi^n (\cdot,z)) ; \quad  \psi(\cdot,z) = K * F(z, \psi(\cdot,z)).$$
with 
 $F(z, x) = \frac12\left(z^2-z\right) +(\rho \nu z- \lambda)  x + \frac{\nu^2}{2} x^2 $.  Hence, for $t \in [0,T]$,
 $$ \psi(t,z) - \psi^n(t,z) = h^n(t,z) + K*\big(F(z,\psi(\cdot,z)) - F(z,\psi^n(\cdot,z))\big)(t), $$
 with $h^n(\cdot,z) = (K^n - K) * F(z, \psi^n(\cdot,z))$. Thanks to Proposition \ref{boundednessRiccati}, we get the existence of a positive constant $C$ such that
 \begin{equation} \label{ineqh_n} 
 \sup_{n \geq 1}\sup_{t \in [0,T]}|h^n(t,a+ib) | \leq C (1+b^4) \int_0^T |K^n(s)-K(s) | ds,
 \end{equation}
 for any $b \in \R $ and $a \in [0,1]$. Moreover notice that $(\psi - \psi^n - h^n)(\cdot,z)$ is solution of the following linear Volterra equation with continuous coefficients
$$ \chi =  K * \left(\big(\rho \nu z - \lambda + \frac{\nu^2}{2}(\psi+\psi_n)(\cdot,z)\big) (\chi + h^n(\cdot,z)) \right), $$
and remark that the real part of $\rho \nu z - \lambda + \frac{\nu^2}{2}(\psi+\psi_n)(\cdot,z) $ is dominated by  $ \nu  - \lambda $ because $\Re(\psi(\cdot, z))$ and $\Re(\psi^n(\cdot, z))$ are non-positive. An application of Corollary \ref{corolC} together with \eqref{ineqh_n} ends the proof.

\subsection{Proof of Proposition \ref{ConvergenceHeston}}
We consider for each $n \geq 1$, $(S^n, V^n)$ defined by the multi-factor Heston model in Definition \ref{multifactorModel} with $\sigma(x) = \nu \sqrt{x}$.
\paragraph{Tightness of $(\int_0^\cdot V^n_s ds, \int_0^\cdot \sqrt{V^n_s} dW_s,  \int_0^\cdot \sqrt{V^n_s} dB_s)_{n \geq 1}$ :}
Because the process $ \int_0^\cdot V^n_s ds$ is non-decreasing, it is enough to show that
\begin{equation} \label{boundForTight}
\sup_{n \geq 1}~ \E[\int_0^T V_t^n dt] < \infty, 
\end{equation}
to obtain its tightness for the uniform topology. Recalling that $ \sup_{t \in [0,T]} \E[V^n_t]  < \infty $ from Proposition \ref{weakExistence} in the Appendix, we get
$$ \E \left[ \int_0^t \sqrt{V_s^n} dB_s\right] =0 ,$$ 
and then by Fubini theorem
$$ E[V^{n}_t] = g^n(t) + \sum_{i = 1}^n c_i^n \E[V^{n,i}_t] ,$$
with 
$$ \E[V^{n,i}_t]  = \int_0^t ( - \gamma_i^n \E[V^{n,i}_s] - \lambda \E[V^{n}_s] )ds. $$
Thus $ t \mapsto \E[V^n_t]$ solves the following linear Volterra equation
$$ \chi(t) =   \int_0^t K^n(t-s) \left( -\lambda \chi(s)+ \theta(s)  + V_0 \frac{s^{-H-\frac{1}{2}}}{\Gamma(1/2-H)} \right) ds , $$
with $K^n$ given by \eqref{smoothedKernel}. Theorem \ref{volterraLinear} in the Appendix leads to
$$ \E[V^n_t] = \int_0^t E^n_\lambda(t-s)  \left(\theta(s)  + V_0 \frac{s^{-H-\frac{1}{2}}}{\Gamma(\frac12-H)} \right) ds, $$
and then by Fubini theorem again
$$ \int_0^t \E[V^n_s] ds = \int_0^t \left(\int_0^{t-s}E^n_\lambda(u) du\right) \left(\theta(s)  + V_0 \frac{s^{-H-\frac{1}{2}}}{\Gamma(\frac12-H)}\right) ds, $$
where $E^n_\lambda$ is the canonical resolvent of $K^n$ with parameter $\lambda$, defined in Appendix \ref{resolvent2}. Because $(K^n)_{n \geq 1}$ converges to the fractional kernel $K$ in $\mathbb L^1([0,T], \R)$, we obtain the convergence of $E^n_\lambda$ in $\mathbb L^1([0,T], \R)$ to the canonical resolvent of $K$ with parameter $\lambda$, see \cite[Theorem 2.3.1]{GLS90}. 
In particular thanks to Corollary \ref{resolvent_positive} in the Appendix, $\int_0^t E^n_\lambda(s) ds $ is uniformly bounded in $t \in [0,T]$ and $n \geq 1$. This leads to \eqref{boundForTight} and then to the tightness of $(\int_0^\cdot V^n_s ds, \int_0^\cdot \sqrt{V^n_s} dW_s,  \int_0^\cdot \sqrt{V^n_s} dB_s)_{n \geq 1}$ by \cite[Theorem VI-4.13]{jacod2013limit}.

\paragraph{Convergence of $(S^n, \int_0^\cdot V^n_s ds)_{n \geq 1}$ :} We set $M^{n, 1}_t =  \int_0^t \sqrt{V^n_s} dW_s $ and $M^{n, 2}_t =  \int_0^t \sqrt{V^n_s} dB_s $. Denote by $(X, M^1, M^2)$ a limit in law for the uniform topology of a subsequence of the tight family $(\int_0^\cdot V^n_s ds, M^{n, 1}, M^{n,2})_{n \geq 1}$. An application of stochastic Fubini theorem, see \cite{V:12}, yields
\begin{equation} \label{equali_n}
 \int_0^t V^n_s ds =  \int_0^t \int_0^{t-s}( K^n(u) - K(u)) du  dY^n_s + \int_0^t K(t-s) Y^n_s ds,  \quad t \in [0,T],
 \end{equation}
where $Y^n_t = \int_0^t (s^{-H-\frac12} \frac{V_0}{\Gamma(1/2-H)} +\theta(s)  - \lambda V^n_s ) ds + \nu M^{n,2}_t $. Because $(Y^n)_{n \geq 1}$ converges in law for the uniform topology to $Y =(Y_t)_{t \leq T}$ given by $Y_t = \int_0^t (s^{-H-\frac12} \frac{V_0}{\Gamma(\frac12-H)} +\theta(s)) ds  - \lambda X_t + \nu M^{2}_t$, we also get the convergence of $ (\int_0^\cdot K(\cdot-s) Y^n_s ds)_{n \geq 1}$ to  $\int_0^\cdot K(\cdot-s) Y_s ds$. Moreover, for any $t \in [0,T]$,
$$ \left| \int_0^t \int_0^{t-s}  (K^n(u) - K(u)) du  \left(s^{-H-\frac12} \frac{V_0}{\Gamma(\frac12-H)} +\theta(s)  - \lambda V^n_s \right) ds \right| $$
is bounded by 
$$ \int_0^t |K^n(s) - K(s)| ds \left( \int_0^t (s^{-H-\frac12} \frac{V_0}{\Gamma(\frac12-H)} +\theta(s)) ds + \lambda \int_0^t V^n_s ds \right),  $$
which converges in law for the uniform topology to zero thanks to the convergence of $(\int_0^\cdot V^n_s ds)_{n \geq 1}$ together with Proposition \ref{convergenceK}. Finally,
$$  \mathbb E \left[ \left|\int_0^t \int_0^{t-s}  (K^n(u) - K(u)) du  dM^{n,2}_s \right|^2  \right] \leq c \int_0^T (K^n(s) - K(s))^2 ds \E\left[\int_0^t V_s^n ds \right], $$
 which goes to zero thanks to \eqref{boundForTight} and Proposition \ref{convergenceK}. Hence, by passing to the limit in \eqref{equali_n}, we obtain
 $$ X_t =  \int_0^t K(t-s) Y_s ds , $$
 for any $t \in [0,T]$, almost surely. The processes being continuous, the equality holds on $[0,T]$.
  Then, by the stochastic Fubini theorem, we deduce that $X = \int_0^\cdot V_s ds $, where $V$ is a continuous process defined by
  $$ V_t = \int_0^t K(t-s) dY_s =V_0 + \int_0^t K(t-s) (\theta(s) - \lambda V_s ) ds + \nu \int_0^t K(t-s) dM^2_s . $$
 Furthermore because $(M^{n,1}, M^{n,2})$ is a martingale with bracket 
 $$ \int_0^\cdot V^n_s ds  \begin{pmatrix}
       1    &      \rho        \\
   \rho        &1       
\end{pmatrix} , $$
\cite[Theorem VI-6.26]{jacod2013limit} implies that $(M^1, M^2)$ is a local martingale with the following bracket
$$ \int_0^\cdot V_s ds  \begin{pmatrix}
       1    &      \rho        \\
   \rho        &1       
\end{pmatrix} . $$
By \cite[Theorem V-3.9]{revuz2013continuous}, there exists a  two-dimensional Brownian motion $(\widetilde{W}, \widetilde{B})$ with  $d\langle \widetilde W, \widetilde B \rangle_t= \rho dt$ such that
$$ M^1_t = \int_0^t \sqrt{V_s} d\widetilde{W}_s , \quad M^2_t = \int_0^t \sqrt{V_s} d\widetilde{B}_s , \quad t \in [0,T]. $$
In particular $V$ is solution of the fractional stochastic integral equation in Definition \ref{roughVolModel} with $\sigma(x) = \nu \sqrt{x}$. Because $S^n = \exp(M^{n,1}\ - \frac{1}{2}\int_0^\cdot V^n_s ds )$, we deduce the convergence of $(S^n, \int_0^\cdot V_s^nds)_{n \geq 1}$ to the limit point $(S, \int_0^\cdot V_s ds)$ that displays the rough-Heston dynamics of Definition \ref{roughVolModel}. The uniqueness of such dynamics, see \cite{papier2, ALP17, mytnik2015uniqueness}, enables us to conclude that $(S^n, V^n)_{n \geq 1}$ admits a unique limit point and hence converges to the rough Heston dynamics.

\subsection{Proof of Proposition \ref{callError}} \label{callErrorProof}
We use the Lewis Fourier inversion method, see \cite{lewis2001simple}, to write
$$
C^n(k,T) - C(k ,T) = S_0 \frac{e^{\frac{k}{2}}}{2 \pi}  \int_{b \in \R} \frac{e^{-ibk}}{b^2 + \frac14}  \left(L(T,\frac12+ib) - L^n(T,\frac12+ib) \right) db .
$$
Hence, 
\begin{equation}  \label{lewis}
|C^n(k,T) - C(k ,T)| \leq S_0 \frac{e^{\frac{k}{2}}}{2 \pi}  \int_{b \in \R} \frac{1}{b^2 + \frac14}  \left| L(T,\frac12+ib) - L^n(T,\frac12+ib) \right| db .
\end{equation}
Because  $L(T,z)$ and $L^n(T,z)$ satisfy respectively the formulas \eqref{formula1} and \eqref{formula2} with $g$ and $g^n$ given by
$$ g(t) = \int_0^t K(t-s) \big(V_0 \frac{s^{-H-\frac12}}{\Gamma(1/2-H)} + \theta(s)\big) ds, \quad g^n(t) = \int_0^t K^n(t-s) \big(V_0 \frac{s^{-H-\frac12}}{\Gamma(1/2-H)} + \theta(s)\big) ds,$$
and $\psi(\cdot,z)$ and $\psi^n(\cdot,z)$ solve respectively \eqref{fracRiccati}  and  \eqref{volterraRiccati}, we use the Fubini theorem to deduce that
\begin{equation} \label{formula1bis}
L(T,z) = \exp\left(\int_0^T \psi(T-s,z) \big(V_0 \frac{s^{-H-\frac12}}{\Gamma(1/2-H)} + \theta(s)\big) ds \right),  
\end{equation}
and 
\begin{equation} \label{formula2bis} 
L^n(T,z) = \exp\left(\int_0^T \psi^n(T-s,z) \big(V_0 \frac{s^{-H-\frac12}}{\Gamma(1/2-H)} + \theta(s)\big) ds \right) ,
\end{equation}
with $z = 1/2 + ib$. Therefore, relying on the local Lipschitz property of the exponential function, it suffices to find an upper bound for $\Re(\psi^n(\cdot,z))$ in order to get an error for the price of the call from \eqref{lewis}. This is the object of the next paragraph.
\paragraph{Upper bound of $\Re(\psi^n(\cdot,z))$ :} 
We denote by $\phi_\eta^n(\cdot,b)$ the unique continuous function satisfying the following Riccati Volterra equation
\begin{equation*}
\phi_\eta^n(\cdot,b) = K^n * \left(-b + \eta \phi_\eta^n(\cdot,b) + \frac{\nu^2}{2} \phi_\eta^n(\cdot,b)^2\right) ,
\end{equation*}
with $b \geq 0$ and $\eta, \nu \in \R $. 
\begin{proposition} \label{ineqAuxRic}Fix  $b_0, t_0\geq 0$ and $\eta \in \mathbb R$. The functions $b  \mapsto \phi_\eta^n(t_0,b)$ and $t  \mapsto \phi_\eta^n(t,b_0)$ are non-increasing on $\mathbb R_+$. Furthermore
$$ \phi_\eta^n(t,b) \leq  \frac{1 - \sqrt{1+ 2 b \nu^2 (\int_0^t E_\eta^n(s) ds)^2}}{\nu^2 \int_0^t E_\eta^n(s) ds},  \quad t > 0,$$
where $E_\eta^n $ is the canonical resolvent of $K^n$ with parameter $\eta$ defined in Appendix \ref{resolvent2}.
\end{proposition}
\begin{proof} The claimed monotonicity of $b   \mapsto \phi_\eta^n(t_0,b)$ is directly obtained from Theorem \ref{T:positive_3}. Consider now $h, b_0>0$. It is easy to see that $\Delta_h \phi_\eta^n(\cdot,b_0)$ solves the following Volterra equation
$$ \Delta_h \phi_\eta^n(b_0, t) = \left(\Delta_tK^n * F(\phi_\eta^n(\cdot,b_0))\right)(h) + \left(K^n*F(\Delta_h \phi_\eta^n(\cdot,b_0))\right)(t)$$
with $F(b,x) = -b + \eta x+ \frac{\nu^2}{2} x^2$. Notice that $t \rightarrow -\left(\Delta_tK^n * F(\phi_\eta^n(\cdot,b_0))\right)(h) \in \Gc_K$, defined in Appendix \ref{AppC}, thanks to Theorem \ref{T:positive_3}. $\phi_\eta^n(\cdot,b)  - \Delta_h \phi_\eta^n(\cdot,b)$ being solution of the following linear Volterra integral equation with continuous coefficients,
$$ x(t) = -\left(\Delta_tK^n * F(b, \phi_\eta^n(\cdot,b_0))\right)(h) +\left(K^n * \left( \left( \eta + \frac{\nu^2}{2}(\phi_\eta^n(\cdot,b) + \Delta_h\phi_\eta^n(\cdot,b)) \right) x \right) \right) (t), $$
we deduce its non-negativity using again Theorem \ref{T:positive_3}. Thus, $t \in \mathbb R_+ \rightarrow \phi_\eta^n(t,b_0)$ is non-increasing and consequently $ \sup_{s \in [0,t]} |\phi_\eta(s,b)| =  |\phi_\eta^n(t,b_0)|$ as $\phi_\eta^n(0,b) = 0$. Hence, Theorem \ref{volterraLinear} leads to
$$ \phi_\eta^n(t,b) = \int_0^t E_\eta^n(t-s) (-b + \frac{\nu^2}{2} \phi_\eta^n(s,b)^2)  \leq \int_0^t E_\eta^n(s) ds \left( -b + \frac{\nu^2}{2} \phi_\eta^n(t,b)^2\right) .$$
We end the proof by solving this inequality of second order in $\phi_\eta^n(t,b)$ and using that $\phi_\eta^n$ is non-positive. Notice that $ \int_0^t E_\eta^n(s) ds > 0 $ for each $t >0$, see Corollary \ref{resolvent_positive}.
\end{proof}

\begin{corollary} \label{corolBoundRic}Fix $a \in [0,1]$. We have, for any $t \in (0, T]$ and $b \in \R$,
$$ \sup_{n \geq 1} \Re(\psi^n(t,a+ib)) \leq \frac{1 - \sqrt{1+  (a - a^2 + (1 - \rho^2) b^2) \nu^2 m(t)^2}}{\nu^2 m(t)}  $$ 
where $m(t) = \inf_{n \geq 1} \int_0^t E^n_{\rho \nu a - \lambda}(s) ds > 0 $ for all $t \in (0,T]$ and $E_\eta^n $ is the canonical resolvent of $K^n$ with parameter $\eta$ defined in Appendix \ref{resolvent2}.
\end{corollary}
\begin{proof} Let $r = a - a^2 + (1 - \rho^2) b^2$ and $ \eta = \rho \nu a - \lambda $. $\phi_{\eta}^n(\cdot,r) - \Re(\psi^n(\cdot,a+ib))$ being solution of the following linear Volterra equation with continuous coefficients
$$ \chi = K * \left(\frac{1}{2}\left(\rho b + \nu \Im(\psi^n(\cdot,a+ib))\right)^2 + \left(\rho \nu a - \lambda + \frac{\nu^2}{2} \left(\Re(\psi^n(\cdot,a+ib)) + \phi_\eta(\cdot,r)\right) \right) \chi \right), $$
we use Theorem \ref{T:positive_3} together with Proposition \ref{ineqAuxRic} to get, for all $t \in [0, T]$ and $b \in \R$,
\begin{equation} \label{ineqRePsi_n}
\Re(\psi^n(t,a+ib)) \leq \frac{1 - \sqrt{1+ 2 r \nu^2 (\int_0^t E_\eta^n(s) ds)^2}}{\nu^2 \int_0^t E_\eta^n(s) ds} .
\end{equation}
Moreover for any $t \in [0, T]$, $\int_0^t E_\eta^n(s) ds $ converges as $n$ goes to infinity to $\int_0^t E_\eta(s) ds $ because $K^n$ converges to $K$ in $\mathbb L^1([0,T], \R)$, see  \cite[Theorem 2.3.1]{GLS90}, where $E_\eta$ denotes the canonical resolvent of $K$ with parameter $\eta$. Therefore, $ m(t) = \inf_{n \geq 1} \int_0^t E_\eta^n(s) ds > 0 $, for all $t \in (0,T]$, because $\int_0^t E_\eta(s) ds > 0$ and  $\int_0^t E_\eta^n(s) ds > 0$ for all $n \geq 1$, see Corollary \ref{resolvent_positive}. Finally we end the proof by using \eqref{ineqRePsi_n} together with the fact that
$ x \mapsto \frac{1 - \sqrt{1+ 2 r \nu^2 x^2}}{\nu^2 x} $
is non-increasing on $(0, \infty)$.

\paragraph{End of the proof of Proposition \ref{callError} :} Assume that $|\rho| < 1$ and fix $a = 1/2$. By dominated convergence theorem, 
$$ \int_0^T \frac{1 - \sqrt{1+  (a - a^2 + (1 - \rho^2) b^2) \nu^2 m(T-s)^2}}{\nu^2 m(T-s)} (\theta(s) + V_0 \frac{s^{-H-\frac12}}{\Gamma(\frac{1}{2}-H)}) ds $$
is equivalent to
$$ - |b| \frac{\sqrt{1-\rho^2}}{\nu}\int_0^T (\theta(s) + V_0 \frac{s^{-H-\frac12}}{\Gamma(\frac{1}{2}-H)}) ds, $$
as $b$ tends to infinity. Hence, thanks to Corollary \ref{corolBoundRic}, there exists $C>0$ such that for any $b \in \R$
\begin{equation} \label{lastIneq}
\sup_{n \geq 1}~ \Re(\psi^n(t,a+ib)) \leq C(1-|b|) . 
\end{equation}
Recalling that
$$ \forall z_1, z_2 \in \C \text{ such that $\Re(z_1), \Re(z_2)\leq c$, } \quad |e^{z_1} - e^{z_2}| \leq e^c |z_1 - z_2|   , $$
we obtain 
$$ |L^n(a+ib, T) - L(a+ib, T)| \leq  e^{C(1-|b|)} \! \! \! \!   \sup_{t \in [0,T]} \!\! \!   |\psi^n(t,a+ib) - \psi(t,a+ib)| \int_0^T \! \! \!  (\theta(s) + V_0 \frac{s^{-H-\frac12}}{\Gamma(\frac{1}{2}-H)})  ds,  $$
from \eqref{formula1bis}, \eqref{formula2bis} and \eqref{lastIneq}.
We deduce Proposition \ref{callError} thanks to \eqref{lewis} and Theorem \ref{convergenceRiccati} together with the fact that $ \int_{b \in \R} \frac{b^4+1}{b^2 + \frac14} e^{C(1-|b|)} db < \infty. $
\end{proof}

\section*{Acknowledgments}

We thank Bruno Bouchard, Christa Cuchiero, Philipp Harms and Mathieu Rosenbaum for many interesting discussions. Omar El Euch is thankful for the support of the Research Initiative ``Mod\'elisation des march\'es actions, obligations et d\'eriv\'es'', financed by HSBC France, under the aegis of the Europlace Institute of Finance.

\section*{Appendix}

\appendix

\section{Stochastic convolutions and resolvents} \label{ConvolComputations}
We recall in this Appendix the framework and notations introduced in \cite{ALP17}.
\subsection{Convolution notation}\label{Convol} 
	
	For a measurable function $K$ on $\R_+$ and a measure $L$ on $\R_+$ of locally bounded variation, the convolutions $K*L$ and $L*K$ are defined by
	$$
	(K*L)(t) = \int_{[0,t]} K(t-s)L(ds), \qquad (L*K)(t) = \int_{[0,t]} L(ds)K(t-s)
	$$
	whenever these expressions are well-defined. If $F$ is a function on $\R_+$, we write $K*F=K*(Fdt)$, that is
	$$ (K*F)(t) = \int_0^t K(t-s) F(s) ds. $$
	We can show that $L * F$ is almost everywhere well-defined and belongs to $\mathbb L^p_{loc}(\R_+, \R)$, whenever $F \in \mathbb L^p_{\rm loc}(\R_+, \R)$. Moreover, $(F * G) * L = F * (G * L) $ {\em a.e.}, whenever $F, G \in \mathbb L^1_{loc}(\R_+, \R)$, see~\cite[Theorem~3.6.1 and Corollary~3.6.2]{GLS90} for further details. \\
	
	For any continuous semimartingale $M = \int _0^. b_s ds + \int_0^. a_s dB_s$  the convolution
	$$ (K*dM)_t = \int_0^t K(t-s)dM_s $$
	is well-defined as an It\^o integral for every $t\ge0$ such that 
	$$\int_0^t |K(t-s)| |b_s| ds + \int_0^t |K(t-s)|^2 |a_s|^2 ds <\infty.$$ 
	Using stochastic Fubini Theorem,  see \cite[Lemma~2.1]{ALP17}, we can show that for each $t \geq 0$, almost surely
	\begin{equation}\label{stochasticFubini} 
	(L*(K*dM))_t = ((L * K) * dM)_t,
	\end{equation}
	whenever $K \in \mathbb L^2_{loc}(\R_+,\R)$ and $a, b$ are locally bounded {\em a.s}. \\
	
	Finally from Lemma 2.4 in \cite{ALP17} together with the Kolmogorov continuity theorem, we can show that there exists a unique version of $(K*dM_t)_{t \geq 0}$ that is continuous whenever $b$ and $\sigma$ are locally bounded. In this paper, we will always work with this continuous version. \\
	
	Note that the convolution notation could be easily extended for matrix-valued $K$ and $L$. In this case, the associativity properties exposed above hold.

\subsection{Resolvent of the first kind} \label{resolvent1}
We define the {\em resolvent of the first kind} of a $d\times d$-matrix valued kernel $K$, as the $\R^{d\times d}$-valued measure $L$ on $\R_+$ of locally bounded variation such that
\begin{equation*}
K*L = L*K \equiv \id,
\end{equation*}
where $\id$ stands for the identity matrix, see~\cite[Definition~5.5.1]{GLS90}. The  resolvent of the first kind does not always exist. In the case of the fractional kernel $K(t) = \frac{t^{H-\frac12}}{\Gamma(H+1/2)}$ the resolvent of the first kind exists and is given by 
$$ L(dt) = \frac{t^{-H-\frac12}}{\Gamma(1/2-H)} dt,$$
for any $H \in (0,1/2)$. If $K$ is non-negative, non-increasing and not identically equal to zero on $\R_+$, the existence of a resolvent of the first kind is guaranteed by \cite[Theorem~5.5.5]{GLS90}. \\

The following result shown in~\cite[Lemma~2.6]{ALP17}, is stated here for $d = 1$ but is true for any dimension $d \geq 1$.  

\begin{lemma}\label{thelemma} Assume that $K \in \mathbb L^1_{\rm loc}(\R_+, \R)$ admits a resolvent of first kind $L$. For any $F \in \mathbb L^1_{loc}(\R_+, \R)$ such that $ F * L $ is right-continuous and of locally bounded variation one has 
	$$ F = (F*L)(0) K+ d(F*L) * K . $$
	Here, $df$ denotes the measure such that $f(t)=f(0)+\int_{[0,t]} df(s)$, for all $t\geq 0$, for any right-continuous function of locally bounde variation $f$ on $\R_+$.
\end{lemma}
\begin{remark} \label{DeltaK}
	The previous lemma will be used with $F = \Delta_h K$, for fixed $h > 0$. If $K$ is continuous on $(0, \infty)$, then $\Delta_h K*L$ is right-continuous. Moreover, if $K$ is  non-negative and  $L$ non-increasing  in the sense that  $s \to  L([s,s+t])$ is non-increasing for all $t\ge0$,  then $\Delta_h K *L$  is non-decreasing  since
	$$ (\Delta_h K * L)(t) = 1 -\int_{ [0,h)} K(h-s) L(t + ds),  \quad  t > 0. $$
	In particular, $\Delta_h K *L$  is of locally bounded variation. 
\end{remark}

\subsection{Resolvent of the second kind} \label{resolvent2}
We consider a kernel $K \in \mathbb L^1_{\rm loc}(\R_+, \R)$ and define {\em the resolvent of the second kind of $K$} as the unique function $R_K \in \mathbb L^1_{\rm loc}(\R_+, \R)$  such that
$$K - R_K =  K *  R_K. $$
For $\lambda \in \R $, we  define {\em the canonical resolvent of $K$  with parameter $\lambda$} as the unique solution $E_\lambda \in \mathbb L^1_{\rm loc}(\R_+, \R)$ of
$$ E_\lambda - K = \lambda K * E_\lambda. $$
This means that $E_\lambda =  - R_{-\lambda K} / \lambda$, when $\lambda \neq 0$ and $E_0=K$.
The existence and uniqueness of $R_K$ and $E_{\lambda}$ is ensured by \cite[Theorem 2.3.1]{GLS90} together with  the continuity of $K \rightarrow E_\lambda(K)$ in the topology of $\mathbb L^1_{\rm loc}(\R_+, \R)$. Moreover, if $K \in \mathbb L^2_{\rm loc}(\R_+, \R)$ so does $E_\lambda$ due to  \cite[Theorem 2.3.5]{GLS90}. \\

We recall \cite[Theorem 2.3.5]{GLS90} regarding the existence and uniqueness of a solution of linear Volterra integral equations in $\mathbb L^1_{\rm loc}(\R_+, \R)$.
\begin{theorem} \label{volterraLinear}Let $f \in \mathbb L^1_{\rm loc}(\R_+, \R)$. The integral equation
$$ x = f + \lambda K * x$$
admits a unique solution $x \in \mathbb L^1_{\rm loc}(\R_+, \R)$ given by
$$ x = f  + \lambda E_\lambda * f . $$
\end{theorem} 

When $K$ and $\lambda$ are positive, $E_\lambda$ is also positive, see \cite[Proposition~9.8.1]{GLS90}.  In that case, we have a Gr\"onwall type inequality given by \cite[Lemma~9.8.2]{GLS90}.
\begin{lemma}\label{gronwall} Let $x, f \in \mathbb L^1_{\rm loc}(\R_+, \R)$ such that 
$$ x(t) \leq (\lambda K* x )(t) + f(t), \quad t \geq 0, \; a.e.$$
Then, 
$$ x(t) \leq f(t)  + (\lambda E_\lambda * f) (t), \quad t \geq 0,  \; a.e.$$
\end{lemma}

Note that the definition of the resolvent of the second kind and canonical resolvent can be extended for matrix-valued kernels. In that case, Theorem \ref{volterraLinear} still holds.

\begin{remark} \label{resolventFrac} The canonical resolvent of the fractional kernel $K(t) = \frac{t^{H-\frac12}}{\Gamma(H+1/2)}$ with parameter $\lambda$ is given by
$$ t^{\alpha - 1} E_\alpha(-\lambda t^\alpha), $$
where $ E_\alpha(x) = \sum_{k \geq 0} \frac{x^k}{\Gamma(\alpha(k+1))} $ is the Mittag-Leffler function and $ \alpha = H + 1/2$ for $H \in (0,1/2)$. 
\end{remark}

	\section{Some existence results for stochastic Volterra equations}
	We collect in this Appendix existence results for general stochastic Volterra equations as introduced in \cite{ALP17}. We refer to \cite{papier2,ALP17} for the proofs.
We fix $T > 0$ and consider the $d$-dimensional stochastic Volterra equation
	\begin{align}\label{E:SVEmulti}
	X_t = g(t) + \int_0^t K(t-s) b(X_s) ds + \int_0^t K(t-s) \sigma(X_s) dB_s, \quad t \in[0, T], 
	\end{align}
		where $b: \mathbb R^d \mapsto \mathbb R^d$, $\sigma : \mathbb R^d \mapsto {\mathbb R}^{d \times m}$ are continuous functions with linear growth, $K \in \mathbb{L}^2 ([0, T], {\mathbb R}^{d \times d})$ is a kernel admitting a resolvent of the first kind $L$,  $g : [0, T] \mapsto \mathbb R^d$ is a continuous function and $B$ is a $m$-dimensional Brownian motion on a  filtered probability space $(\Omega, {\cal F}, \mathbb F, \mathbb P)$. 
	In order to prove the weak existence of continuous solutions to \eqref{E:SVEmulti}, the following regularity assumption is needed.
	\begin{assumption} \label{regularity} There exists $\gamma > 0$ and $C>0$ such that for any $t, h \geq 0 $ with $t + h \leq T$,
	$$ |g(t+h) - g(t)|^2 + \int_0^h |K(s)|^2 ds + \int_0^{T-h} |K(h+s) - K(s)|^2 ds  \leq C h^{2\gamma}.$$
	\end{assumption}

	The following existence result can be found in \cite[Theorem A.1]{papier2}.
	\begin{proposition}\label{weakExistence} Under Assumption \ref{regularity}, the stochastic Volterra equation \eqref{E:SVEmulti}  admits a weak continuous solution $X=(X_t)_{t \leq T}$. Moreover $X$ satisfies 
	\begin{equation} \label{sup_X}
	\sup_{t \in [0,T]} \E[|X_t|^p] < \infty, \quad p > 0, 
	\end{equation}
	and admits H\"older continuous paths on $[0, T]$ of any order strictly less than $\gamma$.
	\end{proposition}
	In particular, for the fractional kernel, Proposition \ref{weakExistence} yields the following result.
	\begin{corollary}\label{weakExistenceFrac} Fix $H \in (0, 1/2)$ and  $\theta : [0,T] \mapsto \R$  satisfying 
	$$  \forall \varepsilon > 0 , \quad \exists C_\varepsilon>0;  \quad \forall u \in (0,T] \quad  |\theta(u)| \leq C_\varepsilon u^{-\frac12 - \varepsilon}. $$
	The fractional stochastic integral equation 
	$$ X_t = X_0 +\frac{1}{\Gamma(H+1/2)} \int_0^t (t-u)^{H - \frac12} (\theta(u) + b(X_u))  du  + \frac{1}{\Gamma(H+1/2)} \int_0^t (t-u)^{H - \frac12}  \sigma(X_u)  dB_u , $$
	admits a  weak continuous solution $X=(X_t)_{t \le T}$ for any $X_0 \in \R$. Moreover $X$ satisfies \eqref{sup_X} and admits H\"older continuous paths on $[0, T]$ of any order strictly less than $H$.
	\end{corollary}
	
	\begin{proof} It is enough to notice that the fractional stochastic integral equation is a particular case of \eqref{E:SVEmulti} with $d =m= 1$, $K(t) = \frac{t^{H-\frac12}}{\Gamma(H+1/2)}$ the fractional kernel, which admits a resolvent of the first kind, see Section \ref{resolvent1}, and 
	$$ g(t) = X_0 + \frac{1}{\Gamma(1/2+H)} \int_0^t (t-u)^{H - 1/2} \theta(u) du. $$
	As $t \mapsto t^{1/2+\varepsilon} \theta(t) $ is bounded on $[0,T]$, we may show that $g$ is $H - \varepsilon$ H\"older continuous for any $\varepsilon > 0$. Hence, Assumption \ref{regularity} is satisfied and the claimed result is directly obtained from Proposition \ref{weakExistence}.
	\end{proof}
	
	We now establish the strong existence and uniqueness of \eqref{E:SVEmulti} in the particular case of smooth kernels. This is done by extending the Yamada-Watanabe pathwise uniqueness proof in \cite{YW71}.
\begin{proposition}\label{T:Yamada} Fix $m=d=1$ and assume that $g$ is H\"older continuous, $K \in C^1([0,T], \R)$ admitting a resolvent of the first kind and that there exists $C>0$ and $\eta \in [1/2,1]$ such that for any $x,y \in \R$,
$$ |b(x) - b(y) | \leq C|x-y|, \quad |\sigma(x) - \sigma(y) | \leq C|x-y|^\eta. $$
Then, the stochastic Volterra equation~\eqref{E:SVEmulti} admits  a unique strong {continuous} solution.
	\end{proposition}
	
	\begin{proof} We start by noticing that, $K$ being smooth, it satisfies Assumption~\ref{regularity}. Hence, the existence of a weak continuous solution to~\eqref{E:SVEmulti} follows from Proposition~\ref{weakExistence}. It is therefore enough to show the pathwise uniqueness. We may proceed similarly to \cite{YW71} by considering $a_0 = 1$, $a_{k - 1} > a_{k }$ for $k \geq 1$ with $\int_{a_{k}}^{a_{k - 1}} x^{-2 \eta} dx = k $ and $\varphi_k \in C^2(\mathbb R, \mathbb R)$  such that $\varphi_k(x) = \varphi_k(-x)$, $\varphi_k(0) = 0$ and for $x > 0$
	\begin{itemize}
	\item $\varphi_k'(x) = 0$ for $x \leq a_k$, $\varphi_k'(x) = 1$ for $x \geq a_{k-1}$ and $\varphi_k'(x) \in [0, 1]$ for $ a_k<x<a_{k - 1}$.
	\item $\varphi_k''(x) \in [0, \frac{2}{k} x^{-2 \eta}]$  for $ a_k<x<a_{k - 1}$.
	\end{itemize}
	Let $X^1$ and $X^2$ be two solutions of \eqref{E:SVEmulti} driven by the same Brownian motion $B$. Notice that, thanks to the smoothness of $K$, $X^i - g$ are semimartingales and for $i = 1, 2$
		$$ d(X^i_t - g(t)) = K(0) dY^i_t +  {(K'*dY^i)_t} \; dt , $$
	with $Y^i_t = \int_0^t b(X^i_s)ds + \int_0^t \sigma(X^i_s)dB_s$.
	Using It\^o's formula, we write
	$$ \varphi_k(X^2_t - X^1_t) = I^1_t + I^2_t + I^3_t,  $$
	where 
	$$ I^1_t = K(0) \int_0^t  \varphi_k'(X^2_s - X^1_s) d(Y^1_s-Y^2_s), $$
	$$ I^2_t =  \int_0^t  \varphi_k'(X^2_s - X^1_s) (K' *d(Y^1-Y^2))_s ds, $$
	$$ I^3_t = \frac{K(0)^2}{2} \int_0^t  \varphi_k''(X^2_s - X^1_s) (\sigma(X^2_s) - \sigma(X^1_s))^2ds . $$
	Recalling that $\sup_{t \leq T} \mathbb E[(X^i_t)^2] < \infty $ for $i = 1, 2$ from Proposition \ref{weakExistence}, we obtain that 
	$$ \E[I^1_t] \leq \E[K(0) \int_0^t  |b(X^2_s) - b(X^1_s)| ds] \leq c \int_0^t \E[|X^2_s - X^1_s|] ds,   $$
	and
	$$ \E[I^2_t] \leq  {c} \int_0^t  \E[(|K'| * |b(X^2) - b(X^1)|)_s] ds \leq  c \int_0^t \E[|X^2_s - X^1_s|] ds,   $$
	because $b$ is Lipschitz continuous and $K'$ is bounded on $[0,T]$. Finally by definition of $\varphi_k$ and the $\eta$-H\"older continuity of $\sigma$, we have
	$$ \E[I^3_t] \leq \frac{c}{k} , $$
	which goes to zero when $k$ is large. Moreover $\E[\varphi_k(X^2_t - X^1_t)]$ converges to $\E[|X^2_t - X^1_t|]$ when $k$ tends to infinity, thanks to the monotone convergence theorem. Thus, we pass to the limit and obtain
	$$ \E[|X^2_t - X^1_t|] \leq c \int_0^t  \E[|X^2_s - X^1_s|] ds . $$
	Gr\"onwall's lemma leads to $  \E[|X^2_t - X^1_t|]  = 0$ {yielding the claimed pathwise uniqueness}.
	
	\end{proof}

	Under additional conditions on $g$ and $K$ one can obtain the existence of non-negative solutions to \eqref{E:SVEmulti} in the case of $d =m= 1$. As in \cite[Theorem 3.5]{ALP17}, the following assumption is needed.
		\begin{assumption} \label{positiveAssump} We assume that $K \in \mathbb L^2([0,T], \R)$ is non-negative, non-increasing and continuous on $(0, T]$. We also assume that its resolvent of the first kind $L$ is non-negative and non-increasing in the sense that $0\le L([s,s+t])\le L([0,t])$ for all $s,t\ge0$ with $s+t \leq T$.
	\end{assumption}
	In \cite{papier2}, the proof of \cite[Theorem 3.5]{ALP17} is adapted to prove the existence of a non-negative solution for a wide class of admissible input curves $g$ satisfying\footnote{Under Assumption \ref{positiveAssump} one can show  that $\Delta_hK * L$ is non-increasing {and right-continuous} thanks to Remark \ref{DeltaK} so that  the associated measure $d(\Delta_hK*L)$ is well-defined.}  
\begin{equation} \label{Croissance}
	\Delta_h g - (\Delta_hK * L)(0) g - d(\Delta_hK * L) * g \geq 0, \quad h \geq 0.
	\end{equation}
	We therefore define the following set of admissible input curves
	\begin{equation*}
	\Gc_K = \left\{g : [0,T] \mapsto \R  \mbox{ continuous satisfying } \eqref{Croissance} \mbox{ and } g(0) \geq 0  \right\}.
	\end{equation*}
	The following existence theorem is a particular case of \cite[Theorem A.2]{papier2}.
	\begin{theorem} \label{existencePositive}
		Assume that $d = m=1$ and that $b$ and $\sigma$ satisfy the boundary conditions
		$$ b(0)\ge0 , \quad  \sigma(0)=0.$$
		Then, under Assumptions \ref{regularity}, and \ref{positiveAssump},  the stochastic Volterra equation \eqref{E:SVEmulti} admits a non-negative weak solution for any $g \in \Gc_K$. 
	\end{theorem}
	
	\begin{remark} \label{positiveFrac} Note that any locally square-integrable completely monotone kernel~\footnote{A kernel $K \in \mathbb L^2_{\rm loc}(\R_+, \R)$ is said to be completely monotone, if it is infinitely differentiable on $(0, \infty)$ such that $(-1)^j K^{(j)} (t) \geq 0$ for any $t >0 $ and $j \geq 0$.} that is not identically zero satisfies Assumption \ref{positiveAssump}, see \cite[Example 3.6]{ALP17}. In particular, this is the case for
	\begin{itemize}
	\item the fractional kernel $K(t) = \frac{t^{H-1/2}}{\Gamma(H+1/2)}$, with $H \in (0,1/2)$.
	\item any weighted sum  of exponentials $K(t) = \sum_{i=1}^n c_i e^{-\gamma_i t}$ such that $c_i,\gamma_i \geq 0$ for all $i \in \{1,\dots, n\}$ and $c_i>0$ for some $i$.
	\end{itemize}
	\end{remark}
	
	\begin{remark} \label{postiiveG_rem}
	Theorem \ref{existencePositive} will be used with functions $g$ of the following form
	$$ g(t) =  c + \int_0^t K(t-s) \xi(ds), $$
	where $\xi$ is a non-negative measure of locally bounded variation and $c$ is a non-negative constant. In that case, we may show that \eqref{Croissance} is satisfied, under Assumption \ref{positiveAssump}.
	\end{remark}

\section{Linear Volterra equation with continuous coefficients}\label{AppC}
In this section, we consider $K\in \mathbb L^2_{\rm loc}(\R_+, \R)$ satisfying Assumption \ref{positiveAssump} with $T = \infty$ and recall the definition of $\Gc_K$, that is
\begin{equation*}
	\Gc_K = \left\{g : \R_+ \mapsto \R  \mbox{ continuous satisfying } \eqref{Croissance} \mbox{ and } g(0) \geq 0  \right\} .
	\end{equation*}
We denote by $\|.\|_{\infty, T}$ the usual uniform norm on $[0,T]$, for each $T>0$.

\begin{theorem} \label{T:positive_3}
Let $K \in \mathbb L^2_{\rm loc}(\R_+, \R)$ satisfying Assumption \ref{positiveAssump} and $g, z, w : \R_+ \mapsto \R $ be continuous functions. The linear Volterra equation 
\begin{equation}\label{E:chi}
\chi =  g+ K*\left( z\chi  + w\right)
\end{equation}
admits a unique continuous solution $\chi$. Furthermore if $g \in \Gc_K$ and $w$ is non-negative, then  $\chi$ is non-negative and
$$ \Delta_{t_0} \chi = g_{t_0} + K * (\Delta_{t_0}z  \Delta_{t_0}\chi + \Delta_{t_0} w) $$
with $g_{t_0}(t) = \Delta_{t_0} g(t) + (\Delta_t K * (z\chi + w))(t_0) \in \Gc_K$, for all  for $t_0,t \geq  0$.

\end{theorem}

\begin{proof} The existence and uniqueness of such solution in $\chi \in \mathbb L^1_{\rm loc}(\R_+, \R)$ is obtained from \cite[Lemma C.1]{ALP17}. Because $\chi$ is solution of \eqref{E:chi},  it is enough to show the local boundedness of $\chi$ to get its continuity.  This follows from Gr\"onwall's Lemma \ref{gronwall} applied on the following inequality
$$ |\chi(t)| \leq \|g\|_{\infty, T} + \left(K*( \|z\|_{\infty, T} |\chi|(.) + \|w\|_{\infty, T})\right)(t), $$
for any $t \in [0,T]$ and for a fixed $T>0$. \\

We assume now that $g \in \Gc_K$ and $w$ is non-negative. The fact that $g_{t_0} \in \Gc_K$, for $t_0 \geq 0$, is proved by adapting the computations of the proof of \cite[Theorem 3.1]{papier2} with $\nu = 0$ provided that $\chi$ is non-negative. In order to establish  the non-negativity of  $\chi$, we introduce, for each $\varepsilon > 0$,  $\chi_{\varepsilon}$ as the unique continuous solution of
\begin{equation}\label{E:chieps}
\chi_{\varepsilon} =  g+ K*\left( z\chi_{\varepsilon} + w +  \varepsilon \right).
\end{equation}
It is enough to prove that $\chi_{\varepsilon}$ is non-negative, for every $\varepsilon >0$, and that $(\chi_{\varepsilon})_{\varepsilon >0}$ converges uniformly on every compact to $\chi$ as $\varepsilon$ goes to zero.

\paragraph{Positivity of $\chi_\varepsilon$ :}It is easy to see that $\chi_{\varepsilon}$ is non-negative on a neighborhood of zero because, for small $t$,
$$ \chi_{\varepsilon}(t) =  g(t) + \left( z(0) g(0) + w(0) + \varepsilon \right)  \int_{0}^{t} K(s)ds  + o(\int_{0}^{t} K(s)ds) ,$$
as $\chi, z$ and $w$ are continuous functions. Hence, $t_0 = \inf \{ t > 0;\quad \chi_\varepsilon(t)<0 \}$ is positive. If we assume that $t_0 < \infty$, we get $\chi_{\varepsilon}(t_0) = 0$ by continuity of $\chi_{\varepsilon}$. $\chi_\varepsilon$ being the solution of \eqref{E:chieps}, we have
$$ \Delta_{t_0} \chi_{\varepsilon} = g_{t_0, \varepsilon} + K * (\Delta_{t_0}z \Delta_{t_0}\chi_{\varepsilon} + \Delta_{t_0}w +  \varepsilon),$$
with $g_{t_0, \varepsilon}(t)  = \Delta_{t_0} g(t) + (\Delta_t K * (z\chi_{\varepsilon} + w + \varepsilon))(t_0) $. Then, by using Lemma \ref{thelemma} with $F = \Delta_tK$, we obtain
\begin{align*} 
g_{t_0, \varepsilon}(t) &= \Delta_{t_0} g(t) - (d(\Delta_t K * L) * g)(t_0) - (\Delta_t K * L) (0) g(t_0) \\
&+ (d(\Delta_t K * L) * \chi_{\varepsilon})(t_0) + (\Delta_t K * L) (0) \chi_{\varepsilon}(t_0) , \end{align*}
which is continuous and non-negative, because $g \in \Gc_K$ and $\Delta_tK*L$ is non-decreasing for any $t\geq 0$, see Remark \ref{DeltaK}. Hence, in the same way, $\Delta_{t_0} \chi_{\varepsilon}$ is non-negative on a neighborhood of zero. Thus $t_0 = \infty$, which means that $\chi_\varepsilon$ is non-negative.

\paragraph{Uniform convergence of $\chi_\varepsilon$ :} We use the following inequality
$$ |\chi-\chi_\varepsilon|(t) \leq \left(K * (\| z \|_{\infty, T}  |\chi-\chi_\varepsilon|  + \varepsilon) \right) (t), \quad t \in [0, T],$$
together with the Gronwall Lemma \ref{gronwall} to show the uniform convergence on $[0, T]$ of $\chi_{\varepsilon}$ to $\chi$ as $\varepsilon$ goes to zero. In particular, $\chi$ is also non-negative. 

\end{proof}

\begin{corollary} \label{resolvent_positive}  Let $K\in \mathbb L^2_{\rm loc}(\R_+, \R)$ satisfying Assumption \ref{positiveAssump} and define $E_\lambda$ as the canonical resolvent of $K$ with parameter $\lambda  \in \R-\{0\}$. Then, $t  \mapsto \int_0^t E_\lambda(s) ds$ is non-negative and non-decreasing on $\R_+$. Furthermore  $\int_0^t E_\lambda(s) ds$ is positive, if $K$ does not vanish on $[0,t]$\end{corollary}
\begin{proof} The non-negativity of $\chi = \int_0^\cdot E_\lambda(s) ds$ is obtained from Theorem \ref{T:positive_3} and from the fact that $\chi$ is solution of the following linear Volterra equation
$$\chi = K * (\lambda \chi + 1) ,$$
by Theorem \ref{volterraLinear}. For fixed $t_0>0$, $\Delta_{t_0} \chi$ satisfies
$$ \Delta_{t_0}\chi = g_{t_0} + K * (\lambda \Delta_{t_0}\chi + 1), $$
with $g_{t_0}(t) = \big( \Delta_t K * (\lambda \Delta_{t_0}\chi + 1)\big) (t_0) \in \Gc_K$, see Theorem \ref{T:positive_3}. It follows that $\Delta_{t_0} \chi - \chi$ solves
$$ x = g_{t_0} + K * (\lambda x).$$
Hence, another application of Theorem \ref{T:positive_3} yields that $\chi \leq \Delta_{t_0} \chi $, proving that $t \to \int_0^t E_\lambda(s)ds$ is non-decreasing.
\end{proof}

We  now provide a version of Theorem \ref{T:positive_3} for complex valued solutions.

\begin{theorem} \label{positiveC} Let $z, w : \mathbb R_+ \mapsto \mathbb C $ be continuous functions and $h_0 \in \mathbb C$. The following linear Volterra equation
\begin{equation*}
h = h_0 + K * (z h + w) 
\end{equation*}
admits unique continuous solution $h : \R_+ \mapsto \mathbb C$ such that  
$$ |h(t)| \leq \psi(t) , \quad t \geq 0,$$
where $\psi : \R_+ \mapsto \R $ is the unique continuous solution of 
$$
 \psi = |h_0| + K * (\Re(z) \psi + |w|). 
$$
\end{theorem}

\begin{proof} The existence and uniqueness of a continuous solution is obtained in the same way as in the proof of Theorem \ref{T:positive_3}. Consider now, for each $\varepsilon > 0$, $\psi_\varepsilon$ the unique continuous solution of 
$$
 \psi_\varepsilon = |h_0| + K * (\Re(z) \psi + |w| + \varepsilon) .
$$
As done in the proof of Theorem \ref{T:positive_3}, $\psi_\varepsilon$ converges uniformly on every compact to $\psi$ as $\varepsilon$ goes to zero. Thus, it is enough to show that, for every $\varepsilon > 0$ and $t \geq 0 $,
$$ |h(t)| \leq \psi_\varepsilon(t). $$

We start by showing the inequality in a neighborhood of zero. Because $z,h,w$ and $\psi_\varepsilon$ are continuous, we get, taking $h_0 = 0$,
$$ |h(t)| = |w(0)| \int_{0}^{t} K(s)ds + o(\int_{0}^{t} K(s)ds), \quad \psi_\varepsilon(t) = (|w(0)| + \varepsilon)  \int_{0}^{t} K(s)ds+ o(\int_{0}^{t} K(s)ds), $$
for small $t$. Hence, $ |h| \leq \psi_\varepsilon$ on a neighborhood of zero. This result still holds when $h_0$ is not zero. Indeed in that case, it is easy to show that for $t$ going to zero,
$$ |h(t)|^2 = |h_0|^2 + 2 \Re\big(\overline{h_0}(z(0)h_0 + w(0))\big) \int_{0}^{t} K(s)ds+ o(\int_{0}^{t} K(s)ds),$$
and 
$$ |\psi_\varepsilon(t)|^2 = |h_0|^2 + 2  \big(\Re(z(0)) |h_0|^2 + |w(0)||h_0| + \varepsilon |h_0|)\big) \int_{0}^{t} K(s)ds+ o(\int_{0}^{t} K(s)ds).$$
As $|h_0|$ is now positive, we conclude that $ |h| \leq \psi_\varepsilon$ on a neighborhood of zero by the Cauchy-Schwarz inequality. \\

Hence, $t_0 = \inf \{ t>0 ; \quad \psi_\varepsilon(t) < |h(t)| \} $ is positive. If we assume that $t_0 < \infty$, we would get that $|h(t_0)| = \psi_\varepsilon(t_0)$ by continuity of $h$ and $\psi_\varepsilon$. Moreover,
$$ \Delta_{t_0}h = \phi_h + K*(\Delta_{t_0}z \Delta_{t_0}h + \Delta_{t_0}w), $$
and 
$$ \Delta_{t_0}\psi_\varepsilon = \phi_{\psi_\varepsilon} + K*(\Delta_{t_0}\Re(z) \Delta_{t_0}w + \Delta_{t_0}|w| + \varepsilon).$$
An application of Lemma \ref{thelemma} with $F = \Delta_tK$ for $t >0$, yields
$$ \phi_h(t) = h_0 (1 -( \Delta_t K*L)(t_0)) + (d(\Delta_t K*L) * h)(t_0)  + ( \Delta_t K*L)(0) h(t_0),$$
and
$$ \phi_{\psi_\varepsilon}(t) = |h_0| (1 -( \Delta_t K*L)(t_0)) + (d(\Delta_t K*L) *{\psi_\varepsilon})(t_0)  + ( \Delta_t K*L)(0) |h(t_0)|.$$
Relying on the fact that $d(\Delta_t K*L)$ is a non-negative measure and $\Delta_tK * L \leq 1$, by Remark \ref{DeltaK}, together with the fact that $|h(s)| \leq \psi_\varepsilon(s) $ for $s \leq t_0$, we get that $|\phi_h(t)| \leq \phi_{\psi_\varepsilon}(t)$. We now notice that in the case $h(t_0) = 0$, we have
$$ \Delta_{t_0}h(t) = \phi_h(t) + w(t_0) \int_{0}^{t} K(s)ds + o(\int_{0}^{t} K(s)ds), $$
and 
$$ \Delta_{t_0}\psi_\varepsilon(t) = \phi_{\psi_\varepsilon}(t) +  (|w(t_0)| + \varepsilon)  \int_{0}^{t} K(s)ds+ o(\int_{0}^{t} K(s)ds) ,$$
and in the case $|h(t_0)|>0$, we have
\begin{align*}
|\Delta_{t_0}h(t)|^2 &= 2  \big(\Re(z(t_0)) |h(t_0)|^2 + \Re(w(t_0)) \Re(h(t_0))   + \Im(w(t_0)) \Im(h(t_0)) \big) \int_{0}^{t} K(s)ds  \\
&  \quad \; +|\phi_h(t)|^2 + o(\int_{0}^{t} K(s)ds),\\
\Delta_{t_0} \psi_\varepsilon(t)^2 &=   2 \big(\Re(z(t_0)) |h(t_0)|^2 + |w(t_0)||h(t_0)| + \varepsilon |h(t_0)|)\big) \int_{0}^{t} K(s)ds\\
&  \quad \; +\phi_{\psi_\varepsilon}(t)^2 + o(\int_{0}^{t} K(s)ds),
\end{align*}
for small $t$, thanks to the continuity of $z, w, h, \phi_h, \phi_{\psi_\varepsilon}$ and $\psi_\varepsilon$. In both cases, we obtain that $ |h| \leq \psi_\varepsilon$ on a neighborhood of $t_0$. Therefore $t_0 = \infty$ and for any $t \geq 0$ 
$$ |h(t)| \leq \psi_{\varepsilon}(t).$$
\end{proof}
The following result is a direct consequence of Theorems \ref{T:positive_3} and \ref{positiveC}.
\begin{corollary} \label{corolC}Let $h_0 \in \mathbb C$ and $z, w : \mathbb R_+ \rightarrow \mathbb C$ be continuous functions such that $\Re(z) \leq \lambda$ for some $\lambda \in \R$. We define $h : \mathbb R_+ \rightarrow \mathbb C$ as the unique continuous solution of 
\begin{equation*}
h = h_0 + K * (z h + w) .
\end{equation*}
Then, for any $t \in [0, T]$,
$$ |h(t)| \leq |h_0| +  (\|w\|_{\infty, T} + \lambda |h_0|) \int_0^T E_{\lambda}(s)ds ,$$
where $E_\lambda$ is the canonical resolvent of $K$ with parameter $\lambda$.
\end{corollary}
\begin{proof} From Theorem \ref{positiveC}, we obtain that $|h| \leq \psi_1$, where $\psi_1$ is the unique continuous solution of
$$ \psi_1 = |h_0| +  K * (\Re(z) \psi_1 + |w|). $$
Moreover define $\psi_2$ as the unique continuous solution of
$$ \psi_2 = |h_0| +  K * (\lambda \psi_2 + \|w\|_{\infty, T}) . $$
Then, $\psi_2 - \psi_1$ solves
$$ \chi = K* (\lambda \chi + f), $$
with $f = (\lambda - \Re(z)) \psi_1 +  \|w\|_{\infty, T} - w$, which is a non-negative function on $[0,T]$. Theorem \ref{T:positive_3} now yields
$$ |h| \leq \psi_1 \leq \psi_2. $$
Finally, the claimed bound follows by noticing that, for $t \in [0,T]$,
$$ \psi_2(t) = |h_0| +  (\|w\|_{\infty, T} + \lambda |h_0|) \int_0^t E_{\lambda}(s)ds,  $$
by Theorem \ref{volterraLinear} and that $\int_0^\cdot E_{\lambda}(s)ds $ is non-decreasing by Corollary \ref{resolvent_positive}.
\end{proof}

 \bibliographystyle{abbrv}
  \bibliography{BibAJEE}

\begin{thebibliography}{10}

\bibitem{papier2}
E.~{Abi Jaber} and O.~{El Euch}.
\newblock Markovian structure of the {V}olterra {H}eston model.
\newblock {\em arXiv preprint arXiv:1803.00477}, 2018.

\bibitem{ALP17}
E.~{Abi Jaber}, M.~Larsson, and S.~Pulido.
\newblock Affine {V}olterra processes.
\newblock {\em arXiv preprint arXiv:1708.08796}, 2017.

\bibitem{bayer2016pricing}
C.~Bayer, P.~Friz, and J.~Gatheral.
\newblock Pricing under rough volatility.
\newblock {\em Quantitative Finance}, 16(6):887--904, 2016.

\bibitem{CC98}
P.~Carmona and L.~Coutin.
\newblock Fractional {B}rownian motion and the {M}arkov property.
\newblock {\em Electron. Comm. Probab.}, 3:95--107, 1998.

\bibitem{CCM00}
P.~Carmona, L.~Coutin, and G.~Montseny.
\newblock Approximation of some {G}aussian processes.
\newblock {\em Stat. Inference Stoch. Process.}, 3(1-2):161--171, 2000.
\newblock 19th ``Rencontres Franco-Belges de Statisticiens'' (Marseille, 1998).

\bibitem{carr1999option}
P.~Carr and D.~Madan.
\newblock Option valuation using the fast {F}ourier transform.
\newblock {\em Journal of Computational Finance}, 2(4):61--73, 1999.

\bibitem{diethelm2002predictor}
K.~Diethelm, N.~J. Ford, and A.~D. Freed.
\newblock A predictor-corrector approach for the numerical solution of
  fractional differential equations.
\newblock {\em Nonlinear Dynamics}, 29(1-4):3--22, 2002.

\bibitem{diethelm2004detailed}
K.~Diethelm, N.~J. Ford, and A.~D. Freed.
\newblock Detailed error analysis for a fractional {A}dams method.
\newblock {\em Numerical algorithms}, 36(1):31--52, 2004.

\bibitem{diethelm1998fracpece}
K.~Diethelm and A.~D. Freed.
\newblock The fracpece subroutine for the numerical solution of differential
  equations of fractional order.
\newblock In {\em Forschung und Wissenschaftliches Rechnen 1998}, pages 57--71.
  Gesellschaft f{\"u}r Wisseschaftliche Datenverarbeitung Gottingen, Germany,
  1999.

\bibitem{euch2017roughening}
O.~El~Euch, J.~Gatheral, and M.~Rosenbaum.
\newblock Roughening {H}eston.
\newblock {\em Available at SSRN: https://ssrn.com/abstract=3116887}, 2018.

\bibitem{euch2016characteristic}
O.~El~Euch and M.~Rosenbaum.
\newblock The characteristic function of rough {H}eston models.
\newblock {\em arXiv preprint arXiv:1609.02108}, 2016.

\bibitem{eleuch2016micro}
O.~El~Euch and M.~Rosenbaum.
\newblock The microstructural foundations of rough volatility and leverage
  effect.
\newblock 2016.

\bibitem{euch2017perfect}
O.~El~Euch and M.~Rosenbaum.
\newblock Perfect hedging in rough {H}eston models.
\newblock {\em arXiv preprint arXiv:1703.05049}, 2017.

\bibitem{fukasawa2011asymptotic}
M.~Fukasawa.
\newblock Asymptotic analysis for stochastic volatility: {M}artingale
  expansion.
\newblock {\em Finance and Stochastics}, 15(4):635--654, 2011.

\bibitem{gatheral2014volatility}
J.~Gatheral, T.~Jaisson, and M.~Rosenbaum.
\newblock Volatility is rough.
\newblock {\em Available at SSRN 2509457}, 2014.

\bibitem{GLS90}
G.~Gripenberg, S.-O. Londen, and O.~Staffans.
\newblock {\em Volterra integral and functional equations}, volume~34 of {\em
  Encyclopedia of Mathematics and its Applications}.
\newblock Cambridge University Press, Cambridge, 1990.

\bibitem{HS15}
P.~Harms and D.~Stefanovits.
\newblock Affine representations of fractional processes with applications in
  mathematical finance.
\newblock {\em arXiv preprint arXiv:1510.04061}, 2015.

\bibitem{jacod2013limit}
J.~Jacod and A.~Shiryaev.
\newblock {\em Limit theorems for stochastic processes}, volume 288.
\newblock Springer Science \& Business Media, 2013.

\bibitem{jaisson2016rough}
T.~Jaisson and M.~Rosenbaum.
\newblock Rough fractional diffusions as scaling limits of nearly unstable
  heavy tailed hawkes processes.
\newblock {\em The Annals of Applied Probability}, 26(5):2860--2882, 2016.

\bibitem{lewis2001simple}
A.~L. Lewis.
\newblock A simple option formula for general jump-diffusion and other
  exponential l{\'e}vy processes.
\newblock {\em Available at SSRN 282110}, 2001.

\bibitem{M11}
A.~A. Muravl\"ev.
\newblock Representation of fractal {B}rownian motion in terms of an
  infinite-dimensional {O}rnstein-{U}hlenbeck process.
\newblock {\em Uspekhi Mat. Nauk}, 66(2(398)):235--236, 2011.

\bibitem{mytnik2015uniqueness}
L.~Mytnik and T.~S. Salisbury.
\newblock Uniqueness for {V}olterra-type stochastic integral equations.
\newblock {\em arXiv preprint arXiv:1502.05513}, 2015.

\bibitem{revuz2013continuous}
D.~Revuz and M.~Yor.
\newblock {\em Continuous martingales and Brownian motion}, volume 293.
\newblock Springer Science \& Business Media, 2013.

\bibitem{V:12}
M.~Veraar.
\newblock The stochastic {F}ubini theorem revisited.
\newblock {\em Stochastics}, 84(4):543--551, 2012.

\bibitem{YW71}
T.~Yamada and S.~Watanabe.
\newblock On the uniqueness of solutions of stochastic differential equations.
\newblock {\em J. Math. Kyoto Univ.}, 11:155--167, 1971.

\end{thebibliography}

\end{document}